\documentclass{siamart1116}

\usepackage{amsmath}
\usepackage{amsbsy}
\usepackage{amssymb}
\usepackage{graphicx}
\usepackage{amsfonts}
\usepackage{algorithmic} % <- Preamble
\usepackage{xcolor}
\newtheorem{Theorem}{Theorem}[section]
\newtheorem{Proposition}{Proposition}[section]
\newtheorem{Remark}{Remark}[section]
\newtheorem{Lemma}{Lemma}[section]

\newtheorem{Definition}{Definition}[section]
\newtheorem{algo}{Algorithm}[section]
\newtheorem{Example}{Example}[section]
\normalsize
\newcommand{\bd}{\begin{displaymath}}
\newcommand{\ed}{\end{displaymath}}
\newcommand{\be}{\begin{equation}}
\newcommand{\ee}{\end{equation}}
\newcommand{\bea}{\begin{eqnarray}}
\newcommand{\eea}{\end{eqnarray}}
\newcommand{\bda}{\begin{eqnarray*}}
\newcommand{\eda}{\end{eqnarray*}}
\newcommand{\ba}{\begin{array}}
\newcommand{\ea}{\end{array}}

\begin{document}

\title{\bf The Forward-Backward-Forward Method from continuous and discrete perspective for pseudo-monotone variational inequalities in Hilbert spaces}

\author{
R. I. Bo\c t\thanks{Corresponding author. Faculty of Mathematics,
	University of Vienna,
	Oskar-Morgenstern-Platz 1,
	1090 Vienna,
	Austria, e-mail: radu.bot@univie.ac.at. Research partially supported by FWF (Austrian Science Fund), project I 2419-N32.}
\and 
 E. R. Csetnek\thanks{Faculty of Mathematics,
 	University of Vienna,
 	Oskar-Morgenstern-Platz 1,
 	1090 Vienna,
 	Austria, e-mail: ernoe.robert.csetnek@univie.ac.at. Research supported by FWF (Austrian Science Fund), project P 29809-N32.}
 \and  P. T. Vuong\thanks{Faculty of Mathematics,
	University of Vienna,
	Oskar-Morgenstern-Platz 1,
	1090 Vienna,
	Austria, e-mail: vuong.phan@univie.ac.at. Research supported by FWF (Austrian Science Fund), project I 2419-N32.}
}
\maketitle

\begin{abstract}
Tseng's forward-backward-forward algorithm is a valuable alternative for  Korpelevich's extragradient method  when solving variational inequalities over a convex and closed set governed by monotone and Lipschitz continuous operators, as it requires in every step only one projection operation. However, it is well-known that Korpelevich's method converges and can therefore be used also for solving variational inequalities governed by pseudo-monotone and Lipschitz continuous operators. In this paper, we first associate to a pseudo-monotone variational inequality a forward-backward-forward dynamical system and carry out an asymptotic analysis for the generated trajectories. The explicit time discretization of this system results into Tseng's forward-backward-forward algorithm with relaxation parameters, which we prove to converge also when it is applied to pseudo-monotone variational inequalities. In addition, we show that linear convergence is guaranteed under strong pseudo-monotonicity. Numerical experiments are carried out for pseudo-monotone variational inequalities over polyhedral sets and fractional programming problems.
\end{abstract}

% REQUIRED

% REQUIRED

%% % % % % % % % % % % % % % % % % % % % % % %
\section{Introduction and preliminaries} 
\label{sec:intro}
%% % % % % % % % % % % % % % % % % % % % % % %
In this paper, the object of our investigation is the following variational inequality of \emph{Stampacchia type}:

Find $x^*\in C$ such that
\begin{equation}\label{VariationalInequality}
\left\langle F(x^*), x-x^*\right\rangle   \geq 0 \quad \forall x\in C,
\end{equation}
where $C$ is a nonempty, convex and closed subset of the real Hilbert space $H$, endowed with inner product $\left\langle \cdot, \cdot\right\rangle $ and corresponding norm $\|\cdot\|$, and
$F: H \to H$ is a Lipschitz continuous operator. We abbreviate the problem \cref{VariationalInequality} as VI($F, C$) and denote its solution set by $\Omega$.

Variational inequalities (VIs) are powerful mathematical models which unify important concepts in applied mathematics, like systems of nonlinear equations, optimality
conditions for optimization problems, complementarity problems, obstacle problems, and
network equilibrium problems (see, for instance, \cite{FacchineiPang03, KinderlehrerStampacchia80}). 
In the last decades, various solution methods for solving problems of type VI($F, C$) have been proposed (see \cite{FacchineiPang03, KinderlehrerStampacchia80}). These methods typically require certain monotonicity properties for the operator $F$ (see \cite{KaramardianSchaible90}). 

The most popular algorithm for solving variational inequalities is the so-called projected-gradient method, which generates, for a starting point $x_0 \in H$, a sequence that approaches the solution set $\Omega$ by
$$ x_{n+1}=P_C(x_n - \lambda F(x_n)) \quad \quad \forall n \geq 0,$$
where $P_C$ is the projection operator onto the convex and closed set $C$ and $\lambda$ is a positive stepsize. It is known  that the sequence $(x_n)_{n\geq 0}$ converges, if $F$ is cococercive (inverse strongly monotone) (\cite{Bauschkebook, ZhuMarcotte96}) or $F$ is
strongly (pseudo-) monotone (\cite{FacchineiPang03,Kha}). The projected-gradient method with variable step sizes was proved to convergence also for  variational inequalities governed by (not necessarily single-valued) maximally monotone and paramonotone operators (\cite{BelloCruzIusem}). If $F$ is ``only'' monotone, then $(x_n)_{n\geq 0}$ does not necessarily  convergence  (see \cite{FacchineiPang03} for an example).  Very recently, Malitsky \cite{Malitsky} introduced a modification of the projected-gradient method, called projected-reflected-gradient method, which, for a starting point $x_0 \in H$, reads
$$  x_{n+1}=P_C(x_n - \lambda F(2x_n-x_{n-1})) \quad \quad \forall n \geq 0.$$
The sequence $(x_n)_{n \geq 0}$ is shown to converge to an element in $\Omega$, if $F$ is monotone. Further extensions of this method can be found in \cite{Malitsky2,Malitsky3}.

The  {mostly used} algorithm in the literature to solve variational inequalities governed by Lipschitz continuous and pseudo-monotone operators is Korpelevich's extragradient method (see \cite{Korpelevich}) or variants of it. All these methods share the feature to perform two projections per iteration. Korpelevich's extragradient method generates, for a starting point $x_0 \in H$, a sequence $(x_n)_{n \geq 0}$ approaching the solution set $\Omega$ as follows
$$ 
\begin{cases}
\begin{aligned}
    y_{n}=P_C(x_n - \lambda F(x_n))\\
    x_{n+1}=P_C(x_n - \lambda F(y_n))
\end{aligned} \quad \quad \forall n \geq 0.
\end{cases}
$$
This algorithm was originally introduced for solving monotone VIs in finite dimensional spaces, however, it was shown in \cite[Theorem 12.2.11]{FacchineiPang03} that it converges even when $F$ is a pseudo-monotone operator. In the last years, the extragradient method has attracted a lot of attention from the research community (see, for instance, \cite{CengTeboulleYao, Censor, FacchineiPang03, HarkerPang90, SolodovSvaiter99, SolodovTseng96, Vuong}). In infinite dimensional spaces, Ceng, Teboulle and Yao  proved in \cite{CengTeboulleYao} that, if $F$ is additionally \emph{sequentially weak-to-strong continuous} (which is however not satisfied by the identity operator), then the sequence $(x_n)_{n \geq 0}$ converges weakly to an element in $\Omega$. It was recently proved in \cite{Vuong} that this statement remains true even if the operator $F$ is  \emph{sequentially weak-to-weak continuous}.

A challenging task when designing efficient algorithms for solving variational inequalities is to keep the number of projection operations performed at each iteration as low as possible. Projection operations may be very expensive, in particular when for these no closed formulas are available. Censor, Gibali and Reich proposed in \cite{CensorGibali,Censor}, for a starting point $x_0 \in H$, the following numerical scheme, called subgradient-extragradient method
$$
\begin{cases}
\begin{aligned}
    y_{n}=P_C(x_n-\lambda F(x_n))\\
    x_{n+1}=P_{T_n}(x_n-\lambda F(y_n)) 
\end{aligned} \quad \quad \forall n \geq 0,
\end{cases} 
$$
where 
$$
T_n=\{w \in H: \left\langle x_n-\lambda F(x_n)-y_n, w-y_n \right\rangle \leq 0 \}.
$$ 
The projection onto the half-space $T_n$ can be explicitly given (see, for instance, \cite{Bauschkebook}), thus, the subgradient-extragradient method requires the computation of only one projection per iteration and outperforms from this point of view the extragadient method. {The subgradient-extragradient method converges for monotone VIs (see \cite{Censor}), but also for 
pseudo-monotone VIs  (see \cite{CensorGibali, ThongShehuIyiola})}.

In this paper, we first attach to VI($F, C$) a dynamical system of forward-backward-forward-type (see \eqref{DS}) and carry out a convergence analysis for the generated trajectories to an element in $\Omega$, in the case when $F$ is a pseudo-monotone operator. If $F$ is assumed to be strongly pseudo-monotone, we prove that the trajectory converges exponentially to the unique solution of VI($F, C$). Dynamical systems of forward-backward-forward type were first studied in  \cite{BB18} in the context of approaching the set of the zeros of the sum of a maximally monotone operator and a monotone and Lipschitz continuous operator by continuous trajectories.

The explicit time discretization of \eqref{DS} leads to Tseng's forward-backward-forward algorithm with relaxation parameters (\cite{Tseng2000}). When applied to  the solving of monotone operators, this algorithm, which requires the computation of only one projection per iteration, 
is known to generate a sequence, which weakly converges to a solution of VI($F, C$). {In this paper we show that this convergence result remains true even if $F$ is a pseudo-monotone and \emph{sequentially weak-to-weak-continuous} operator, for both an underrelaxed and an overrelaxed variant of Tseng's algorithm, and 
provide examples of operators that fulfill the assumptions of the convergence theorems. We also prove that the convergence statement remains true  in finite dimensional spaces under less restrictive assumptions on $F$. In addition, we propose an adaptive stepsize strategy, which does not require the knowledge
of the Lipschitz constant of the governing operator. This shows} that Tseng's algorithm is a method to be considered when solving constrained pseudo-convex differentiable
optimization problems. We also show that linear convergence is guaranteed when the pseudo-monotonicity for $F$ is replaced by strong pseudo-monotonicity. In the last section we carry out numerical experiments which show that, when  applied to pseudo-monotone variational inequalities over polyhedral sets and to fractional programming problems, Tseng's method outperforms Korpelevich's extragradient method and even the subgradient-extragradient method.

{We want to notice that a single projection method of Halpern-type for pseudo-monotone variational inequalities in Hilbert spaces, which consequently generates a strongly convergent sequence to a solution, has been recently provided in \cite{ShehuDongJiang}.}

We close this section by recalling some notions and results which will be useful within this paper.

\begin{Definition} Let $C$ be a nonempty subset of the real Hilbert space $H$. The mapping $F: H \rightarrow H$ is said to be
\begin{description}
\item[{\rm (a)}]  pseudo-monotone on $C$, if for every $x,y \in C$ it holds
$$
\langle F(x),y-x\rangle\geq 0\;\Rightarrow\; \langle F(y),y-x\rangle\geq 0;
$$
\item[{\rm (b)}]  monotone on $C$, if  for every $x,y \in C$ it holds
$$
\langle F(y)-F(x), y-x\rangle\geq 0;
$$
\item[{\rm (c)}] $\gamma$-strongly pseudo-monotone on $C$ with $\gamma >0$, if
for every $x,y \in C$ it holds
$$
\langle F(x),y-x\rangle\geq 0\;\Rightarrow\; \langle F(y),y-x\rangle\geq \gamma\|x-y\|^2; 
$$ 
\item[{\rm (d)}]   $\gamma$-strongly monotone on $C$ with $\gamma >0$, if
for every $x,y \in C$ it holds
$$
\langle F(y)-F(x), y-x\rangle\geq \gamma\|x-y\|^2.
$$
\end{description}
\end{Definition}

For a survey on pseudo-monotone operators and their applications in consumer theory of mathematical economics we refer to \cite{HadjisavvasSchaibleWong}.

% \begin{Remark}
% It is clear that (strong) monotonicity implies (strong) pseudo monotonicity.
% However, the converse does not hold. 
% %For example the mapping 
% %$F:\left(0,+\infty \right) \to \left(0,+\infty \right)$, 
% %defined by $F(x)=\frac{a}{a+x}$ with $a>0$ is pseudo-monotone but not monotone.
% \end{Remark}

We recall that the operator $F : H \rightarrow H$ is called \emph{Lipschitz continuous with Lipschitz constant $L > 0$}, if for every $x,y \in H$ it holds
$$
\|F(x)-F(y)\| \leq L \|x-y \|.
$$
The operator $F$ is called \emph{sequential weak-to-weak continuous},
if for every sequence $(x_n)_{n \geq 0}$ that converges weakly to $x$ the sequence $(F(x_n))_{n \geq 0}$ converges weakly to $F(x)$. 

For a nonempty, convex and closed set $C \subseteq H$ and an arbitrary element $x \in H$, there exists a unique element in $C$, denoted by $P_C(x)$, such that
$$
\|x-P_C(x)\|\leq \|x-y\| \quad \forall y\in C.
$$
The operator $P_C : H \rightarrow C$ is the \emph{projection operator} onto $C$. For all $x \in H$ and $y \in C$ it holds
\begin{equation}\label{projectionProperty}
\left<x-P_C(x),y-P_C(x) \right>\leq 0.
\end{equation}
One can also easily see that, for $\lambda >0$, $x^*$ is a solution of {\rm VI($F, C$)} if and only if $x^*=P_C(x^*-\lambda F(x^*))$. We recall the following characterization of the solution set
of a pseudo-monotone variational inequality (\cite[Lemma 2.1]{CottleYao}).

\begin{Proposition}\label{Minty_Lemma}
Let $C$  be a nonempty, convex and closed subset of the real Hilbert space $H$ and $F: H\rightarrow H$ an operator which is
pseudo-monotone on $C$ and continuous. Then for every $x \in C$ we have
\begin{equation}\label{Minty}
\langle F(x), y-x \rangle \geq 0 \quad \forall y \in C \ \Leftrightarrow \ \langle F(y), y-x\rangle\geq 0 \quad \forall y\in C.
\end{equation}
\end{Proposition}
The variational inequality:

Find $x^*\in C$ such that
\begin{equation*}
\left\langle F(x), x-x^*\right\rangle   \geq 0 \quad \forall x\in C
\end{equation*}
is called of \emph{Minty type}. Proposition \ref{Minty_Lemma} shows that the two variationaly inequalities have the same set of solutions when they are formulated over a nonempty, convex and closed set and governed by pseudo-monotone and continuous operators. Existence results for solutions of variational inequalities have been obtained for instance in \cite{KinderlehrerStampacchia80, Laszlo,  MaugeriRaciti}.

%% % % % % % % % % % % % % % % % % % % % % % %
\section{A dynamical system of forward-backward-forward type} \label{Dynamic} 
%%%%%%%%%%%%%%%%%%%%%%%%%%%%%%%%%%%%%%%%%%%%%%%%%%%%%%%
In this section we will approach the solution set of VI($F, C$) from a continuous perspective by means of trajectories generated by the following dynamical system of forward-backward-forward type
\begin{eqnarray}\label{DS}
\begin{cases}
\begin{aligned}
y(t)&=P_C(x(t) - \lambda F(x(t)))\\
\dot{x}(t)+x(t) &=y(t) + \lambda\left[ F(x(t))-F(y(t))\right] \\
x(0)&=x_0,
\end{aligned} 
\end{cases}
\end{eqnarray}
where $\lambda>0$ and $x_0 \in H$. The formulation of \eqref{DS} has its roots in \cite{BB18}, where the continuous counterpart of Tseng's algorithm has been considered in the more general context of a monotone inclusion problem.
The existence and uniqueness of the trajectory $x \in {\cal C}^1([0,+\infty),H)$ generated by \eqref{DS} has been established in \cite{BB18}, as a consequence of the global Cauchy-Lipschitz Theorem and  by making use of the Lipschitz continuity of $F$. Here
we study the convergence of $x(t)$ and $y(t)$ to an element in $\Omega$ as $t \to +\infty$, in the case when $F$ is pseudo-monotone.
	
	\begin{Remark}\label{remark21} {\rm
		The explicit time discretization of the dynamical system (\ref{DS})  with step size $\rho_n>0$ and initial point $x_0 \in H$ yields for every $n \geq 0$ the following equation
		$$
		\frac{x_{n+1}-x_n}{\rho_n} + x_n= P_{C}(x_n-\lambda F(x_n))+\lambda F(x_n)-\lambda F[P_{C}(x_n-\lambda F(x_n))].
		$$
		Denoting $y_n:=P_{C}(x_n-\lambda F(x_n))$, we can rewrite this scheme as
		\begin{equation} \label{RFBF}
		\begin{cases}
		\begin{aligned}
		y_{n}&=P_{C}(x_n-\lambda F(x_n))\\
		x_{n+1}&= \rho_n \left( y_n + \lambda (F(x_n)-F(y_n))\right) +(1-\rho_n)x_n
		\end{aligned} \quad \quad \forall n \geq 0,
		\end{cases} 
		\end{equation}	
		which is precisely Tseng's forward-backward-forward algorithm with relaxation parameters $(\rho_n)_{n \geq 0}$. In the case $\rho_n=1$ for every $n \geq 0$, this iterative scheme reduces to the classical forward-backward-forward algorithm as it was introduced in \cite{Tseng2000}. 
		In Section \ref{sec:TsengMethod} we prove the convergence of the algorithm in \eqref{RFBF}.

}
\end{Remark}	

In the following we will investigate the asymptotic behaviour of the trajectory generated by the dynamical system \eqref{DS}. To this end we will use the following two results. The first one (see \cite[Lemma 5.2]{AAS2014}) is the continuous
counterpart of a result which states the convergence of quasi-Fej\'er monotone
sequences. The  second one (see \cite[Lemma 5.3]{AAS2014}) is the continuous version of the Opial Lemma.

\begin{Lemma} \label{FejerContinuous} 
	If $1\leq p<\infty, 1\leq r < \infty, A: [0,+\infty) \to [0,+\infty)$ is locally absolutely
	continuous, $A \in L^p([0,+\infty))$, $ B: [0,+\infty) \to \mathbb{R}, B \in L^r([0,+\infty))$ and for almost every $t \in [0,+\infty)$
	$$
	\frac{d}{dt}  A(t) \leq B(t),
	$$	
	then $\lim_{t \to +\infty} A(t)=0$.
\end{Lemma}

\begin{Lemma}\label{OpialContinuous} 
	Let $\Omega \subseteq H$ be a nonempty set and $x : [0, +\infty) \to H $ a given map. Assume
	that
	\begin{itemize}
		\item[(i)] for every $x^* \in \Omega$ the limit $\lim_{t \to +\infty} \|x(t)-x^*\|$ exists;
		\item[(ii)] every weak sequential cluster point of the map $x$ belongs to $\Omega$.
	\end{itemize}
	Then there exists an element $x^{\infty} \in \Omega$ such that $x(t)$ converges weakly to $x^{\infty}$ as $t \rightarrow +\infty$.
\end{Lemma}

We start our asymptotic analysis with two preliminary results.

\begin{Proposition}\label{DSProperty} 
	Assume that the solution set $\Omega$ is nonempty, $F$ is pseudo-monotone on $C$ and Lipschitz continuous with constant $L>0$. Then for every  solution $x^* \in \Omega$ it holds
	\begin{equation*}\label{DSdeceasing}
	\left\langle \dot{x}(t), x(t)-x^*\right\rangle  \leq -\left( 1-\lambda L\right) \|x(t)-y(t)\|^2 \leq 0 \quad \forall t \in [0,+\infty).
	\end{equation*}
\end{Proposition}
\begin{proof}
	Since $x^* \in \Omega$ and $y(t) \in C$ it holds	
	$$
	\left\langle F(x^*), y(t)-x^*\right\rangle \geq 0 \quad \forall t \in [0,+\infty).
	$$
	By the pseudo-monotonicity of $F$ it holds
\begin{equation}\label{DS1}
	\left\langle F(y(t)), y(t)-x^*\right\rangle \geq 0 \quad \forall t \in [0,+\infty).
	\end{equation}
	On the other hand, since $y(t)=P_C(x(t) - \lambda F(x(t)))$, we obtain from 	\eqref{projectionProperty} that 
	\begin{equation}\label{DS2}
	\left\langle x(t) - \lambda F(x(t))-y(t), y(t)-x^*\right\rangle \geq 0 \quad \forall t \in [0,+\infty).
	\end{equation}		
Combining \eqref{DS1} and \eqref{DS2} we obtain for every $t \in [0,+\infty)$
	$$
	\left\langle x(t) - y(t)-\lambda \left[ F(x(t))-F(y(t))\right] , y(t)-x^*\right\rangle \geq 0
	$$
	or, equivalently, by taking into account the formulation of the dynamical system \eqref{DS}
	$$
	\left\langle x(t) - y(t)-\lambda \left[ F(x(t))-F(y(t))\right] , y(t)-x(t)\right\rangle 
	- \left\langle \dot{x}(t), x(t)-x^*\right\rangle  \geq 0.
	$$
	This implies that
	\begin{align*}
	\left\langle \dot{x}(t), x(t)-x^*\right\rangle 
	& \ \leq \left\langle x(t) - y(t)-\lambda \left[ F(x(t))-F(y(t))\right] , y(t)-x(t)\right\rangle \\
	& \ =-\|x(t)-y(t)\|^2+\lambda \left\langle  F(x(t))-F(y(t)), x(t)-y(t)\right\rangle\\
	& \ \leq  -\left( 1-\lambda L\right) \|x(t)-y(t)\|^2 \quad \forall t \in [0,+\infty).
	\end{align*}	
\end{proof}

\begin{Proposition}\label{DSProperty2} 
	Assume that the solution set $\Omega$ is nonempty, $F$ is pseudo-monotone on $C$ and Lipschitz continuous with constant $L>0$, and $0 < \lambda < \frac{1}{L}$. Then, for every  solution $x^* \in \Omega$,  the function 
	$t \to \|x(t)-x^*\|^2$ is nonincreasing and it holds
	$$
	\int_{0}^{+\infty} \|x(t)-y(t)\|^2 dt < +\infty \ \mbox{and} \
	\lim_{t \to +\infty} \|x(t)-y(t)\| = 0. 
	$$
\end{Proposition}

\begin{proof}
	Using Proposition \ref{DSProperty}, for every $t \in [0,+\infty)$ it holds
	$$
	\frac{1}{2}\frac{d}{dt} \|x(t)-x^*\|^2=\left\langle x(t)-x^*, \dot{x}(t) \right\rangle 
	\leq  -\left( 1-\lambda L\right) \|x(t)-y(t)\|^2 \leq 0,
	$$
	which shows that $t \to \|x(t)-x^*\|^2$ is nonincreasing. Let be $T >0$. Integrating the previous inequality from $0$ to $T$ it yields 
	$$
	\left( 1-\lambda L\right) \int_{0}^{T}  \|x(t)-y(t)\|^2 dt 
	\leq \frac{1}{2} \left( \|x(0)-x^*\|^2-\|x(T)-x^*\|^2\right) \leq  \frac{1}{2} \|x(0)-x^*\|^2.
	$$
	Letting $T \rightarrow +\infty$, it follows that $\int_{0}^{+\infty} \|x(t)-y(t)\|^2 dt < +\infty$.	
	
Since $P_C$ is nonexpansive and $F$ is Lipschitz continuous with constant $L$, we get that $P_C \circ (I-\lambda F)$ is Lipschitz continuous with constant $1+\lambda L$. Using that
$$y(t) = P_C \circ (I-\lambda F)(x(t)) \quad \forall t \in [0,+\infty),$$
if follows that the trajectory $y$ is locally absolutely continuous and that for almost every $t\in [0,+\infty)$ it holds
	$$
	\|\dot{y}(t)\| \leq (1+\lambda L) \|\dot{x}(t)\|. 
	$$
	On the other hand,
	$$
	\|\dot{x}(t)\| = \|x(t)-y(t)-\lambda\left[ F(x(t))-F(y(t))\right]\| \leq  (1+\lambda L) \|x(t)-y(t)\| \quad \forall t \in [0,+\infty). 
	$$
	Thus, for almost every $t\in [0,+\infty)$,
	\begin{eqnarray*}
		\frac{d}{dt} \|x(t)-y(t)\|^2 
		&=& 2\left\langle x(t)-y(t), \dot{x}(t)-\dot{y}(t)\right\rangle \\
		&\leq& 2\left(\|\dot{x}(t)\| +\|\dot{y}(t)\|  \right) \|x(t)-y(t)\|\\
		&\leq& 2\left(1+\lambda L+(1+\lambda L)^2 \right) \|x(t)-y(t)\|^2.
	\end{eqnarray*}
From here, according to Lemma \ref{FejerContinuous}, we obtain
	\begin{align*}
	\lim_{t \to +\infty} \|x(t)-y(t)\| = 0. 
	\end{align*}
\end{proof}

We come now to the main theorem of this section.
\begin{Theorem}\label{Main result DS} 
	Assume that the solution set $\Omega$ \!is nonempty, \!$F$ \!\! is pseudo-monotone on $H$, Lipschitz continuous with constant $L>0$ and sequentially weak-to-weak continuous, and $0 < \lambda < \frac{1}{L}$. Then the trajectories $x(t)$ and $y(t)$ generated by \eqref{DS} converge  weakly
	to a solution of {\rm VI($F, C$)} as $t \rightarrow +\infty$.
\end{Theorem}
\begin{proof}
	Let $\hat{x} \in H$ be a weak sequential cluster point of $x(t)$ as $t \to +\infty$ and $(t_n)_{n \geq 0}$ be a sequence in $[0,+\infty)$ with $t_n \to +\infty$ and $x(t_n) \rightharpoonup \hat{x}$ as $n \to +\infty$. 
	Since $\lim_{t \to +\infty} \|x(t)-y(t)\| = 0$, we also have $y(t_n) \rightharpoonup \hat{x}$ as $n \to +\infty$. Furthermore, since $F$ is Lipschitz continuous,
	$\|F(x(t_n))-F(y(t_n))\| \to 0$ as $n \to +\infty$. We will prove that $\hat{x} \in \Omega$. 
	For convenience, we denote $x_n := x(t_n)$ and $y_n := y(t_n)$ for every $n \ge 0$. 	
	Since $(y_n)_{n \geq 0} \subseteq C$ and $C$ is weakly closed, we have $\hat{x} \in C$. 
	%Our aim is to prove that $\hat{x} \in \Omega$. 
	We assume that $F(\hat x) \neq 0$, otherwise the conclusion follows automatically.
	
Let $y \in C$ be fixed. For every $n \geq 0$ we have
	$$y_{n} =P_C(x_{n}-\lambda F(x_{n})),$$
	thus
	$$
	\left\langle x_{n}-\lambda F(x_{n})-y_{n}, y-y_{n} \right\rangle \leq 0
	$$
	or, equivalently,
	\begin{equation}\label{ine-1}
	\frac{1}{\lambda}\left\langle x_{n}-y_{n}, y-y_{n} \right\rangle 
	\leq   \left\langle  F(x_{n})-F(y_{n}),y-y_{n}  \right \rangle
	+ \left\langle  F(y_{n}),y-y_{n}  \right \rangle.
	\end{equation}
	Letting  in the last inequality $n \to +\infty$ and  taking into account that
	$\lim_{n \to +\infty}\|x_{n}-y_{n} \|=0$,
	$
	\lim_{n \to +\infty} \|F(x_{n})-F(y_{n})\|=0
	$ and $(y_{n})_{n \geq 0}$ is bounded,
	it follows
	\begin{equation*}\label{ine1}
	\liminf_{n \to +\infty} \left\langle  F(y_{n}),y-y_{n}  \right \rangle \geq 0.
	\end{equation*}
On the other hand, we have that $(y_{n})_{n \geq 0} $ converges weakly to $\hat{x}$ as $n \to +\infty$. Since $F$ is sequentially weak-to-weak continuous, $(F(y_{n}))_{n \geq 0}$ converges weakly to $F(\hat{x})$ as $n \to +\infty$. Since the norm mapping is sequentially weakly lower semicontinuous, we have 
	$$
	0 < \|F(\hat{x}) \| \leq \liminf_{n \to +\infty} \|F(y_{n})\|.
	$$
Then there exists $n_{-1} \geq 0$ such that $F(y_n) \neq 0$ for all $n \geq n_{-1}$.

Let $(\epsilon_k)_{k \geq 0}$ be a positive strictly decreasing sequence which converges to $0$ as $k \rightarrow +\infty$. 

Since $\sup_{N \geq 0} \inf_{n \geq N} \left\langle  F(y_{n}),y-y_{n}  \right \rangle = \liminf_{n \to +\infty} \left\langle  F(y_{n}),y-y_{n}  \right \rangle  > -\epsilon_0,$ there exists $N_0 \geq 0$ such that $\inf_{n \geq N_0} \left\langle  F(y_{n}),y-y_{n}  \right \rangle > -\epsilon_0$. Taking $n_0 > \max\{N_0, n_{-1}\}$, we have
$$\left\langle  F(y_{n_{0}}),y-y_{n_{0}} \right \rangle + \epsilon_0  \geq 0 \ \mbox{and} \ F(y_{n_{0}}) \not= 0.$$
We can continue this construction inductively and assume to this end that $n_0 < n_1 < ... < n_k$ are given. Then there exists $N_{k+1} \geq 0$ such that $\inf_{n \geq N_{k+1}} \left\langle  F(y_{n}),y-y_{n}  \right \rangle > -\epsilon_{k+1}$. Taking $n_{k+1} > \max\{N_{k+1}, n_{k}\}$, we have
$$\left\langle  F(y_{n_{k+1}}),y-y_{n_{k+1}} \right \rangle + \epsilon_{k+1}  \geq 0 \ \mbox{and} \ F(y_{n_{k+1}}) \not= 0.$$
In this way we obtain a strictly increasing sequence $(n_k)_{k \geq 0}$ with the property that 
	\begin{equation}\label{ine2}
	\left\langle  F(y_{n_{k}}),y-y_{n_{k}} \right \rangle + \epsilon_k  \geq 0 \ \mbox{and} \ F(y_{n_{k}}) \not= 0 \quad \forall k \geq 0.
	\end{equation}
	Setting for every $k \geq 0$
	$$
	z_{k}:=\frac{F(y_{n_{k}})}{\|F(y_{n_{k}})\|^2},
	$$
	it holds $ \left\langle  F(y_{n_{k}}),z_{k}  \right \rangle=1$. According to \eqref{ine2} we have that
	$$
	\left\langle  F(y_{n_{k}}),y+\epsilon_k z_{k} -y_{n_{k}}  \right \rangle\geq 0 \quad \forall k \geq 0.
	$$
	Since $F$ is pseudo-monotone on $H$, it yields
	\begin{equation}\label{ine3}
	\left\langle  F(y+\epsilon_k z_{k}), y+\epsilon_k z_{k} -y_{n_{k}}  \right \rangle\geq 0 \quad \forall k \geq 0.
	\end{equation}
Using that $(F(y_{n_{k}}))_{k \geq 0}$ is bounded we have
	$$
	\lim_{k \to +\infty} \|\epsilon_k z_{k}\| = \lim_{k \to +\infty} \frac{\epsilon_k}{\|F(y_{n_k})\|} =0.
	$$
	Taking  in \eqref{ine3} the limit as $k \to +\infty$ we obtain 
	$$
	\left\langle  F(y),y-\hat{x}  \right \rangle \geq 0.
	$$
	As $y$ was arbitrarily chosen in $C$, it follows from Proposition \ref{Minty_Lemma} that $\hat{x} \in \Omega $.
	
On the other hand, by Proposition \ref{DSProperty2}, for every $x^* \in \Omega$, $\|x(t)-x^*\|$ converges as $t \rightarrow +\infty$. Thus, according to
	the Lemma \ref{OpialContinuous}, $x(t)$ converges weakly to an element of
	$\Omega$ as $t \rightarrow +\infty$.
	Since, due to Proposition \ref{DSProperty2}, we have that
	$$
	\lim_{t \to +\infty} \|x(t)-y(t)\| = 0,
	$$
	it follows that $y(t)$ converges weakly to the same element of
	$\Omega$ as $t \rightarrow +\infty$. 	
\end{proof}

The following example introduces a class of operators which are pseudo-monotone, Lipschitz continuous and sequentially weak-to-weak continuous on $H$, but are not necessarily monotone.

\begin{Example} \label{Exam1}\rm
Let $F: H \rightarrow H$  be defined as
\begin{equation*}
	F(x):= g(x)(Mx+p),
	\end{equation*}
	where $M: H \rightarrow H$  is a linear bounded operator satisfying 
	\begin{equation*}\label{Linear_Coercivity}
	\langle Mx, x\rangle\geq 0 \quad \forall x\in H,
	\end{equation*}
	$p \in H$, and $g: H \to (0,+\infty)$ is a function taking positive values. 

Such operators have been considered in \cite{Bianchi} in the case when $H$ is a finite dimensional  and $M$ is a skew operator, i.e. 
	$\left\langle Mx, x\right\rangle = 0$ for every $x\in H$, under the name \emph{pseudo-affine operators}. In general $F$ is not monotone, see \cite{Bianchi}. This fact is reflected by Figure \ref{Fig1}  in the case when $H=\mathbb{R}$.

We show that $F$ is pseudo-monotone on $H$. Indeed, let $x,y\in H$ be such that $\langle F(x), y-x \rangle\geq 0$. Since
	$g(x)>0$, we have
	$$
	\langle Mx+p, y-x\rangle\geq 0.
	$$
	Hence
	\begin{align*}\label{Pseudo-Exam}
	\nonumber \langle F(y), y-x\rangle&= \ g(y)\langle My+p, y-x\rangle \geq g(y)(\langle My+p, y-x\rangle-\langle Mx+p, y-x\rangle)\\
	&= \ g(y)\langle M(y-x), y-x\rangle \geq 0,
	\end{align*}	
which leads to the desired conclusion.

Since every linear bounded operator $M : H \rightarrow H$ is sequentially weak-to-weak continuous, the operator $F$ is sequentially weak-to-weak continuous if $g$ is weakly continuous. This is for instance the case when $g$ has the expression $g(x):=\eta(\left\langle a,x\right\rangle )$ for a fixed vector $a\in H$ and a continuous function $\eta: \mathbb{R} \to (0,+\infty)$.  

In addition, for some choices of $H$,  $a$ and  $\eta$ the operator $F$ is Lipschitz continuous. Indeed, for $H=\ell_2$, $a = e_1= (1,0,0, ...) \in \ell_2$,  $\eta(t) = e^{-t^2}$, 
$$M : \ell_2 \rightarrow \ell_2, M(x_1, x_2, ...) = (x_1, 0, 0, ... ), \ \mbox{and} \ p = 0 \in \ell_2,$$ the operator $F : \ell_2 \rightarrow \ell_2, F(x_1, x_2, ....) = (x_1 e^{-x_1^2}, 0, 0, ...)$, is Lipschitz continuous. This follows by the Mean Value Theorem, since it is easy to see by direct computation that there exists $L >0$ such that $\|\nabla F (x)\| \leq L$  for every $x \in \ell_2$. We illustrate in 
Figure \ref{Fig1} this choice of $F$. 
	
	\begin{figure}[!htbp]
		\begin{centering}
			\includegraphics[width=1\textwidth]{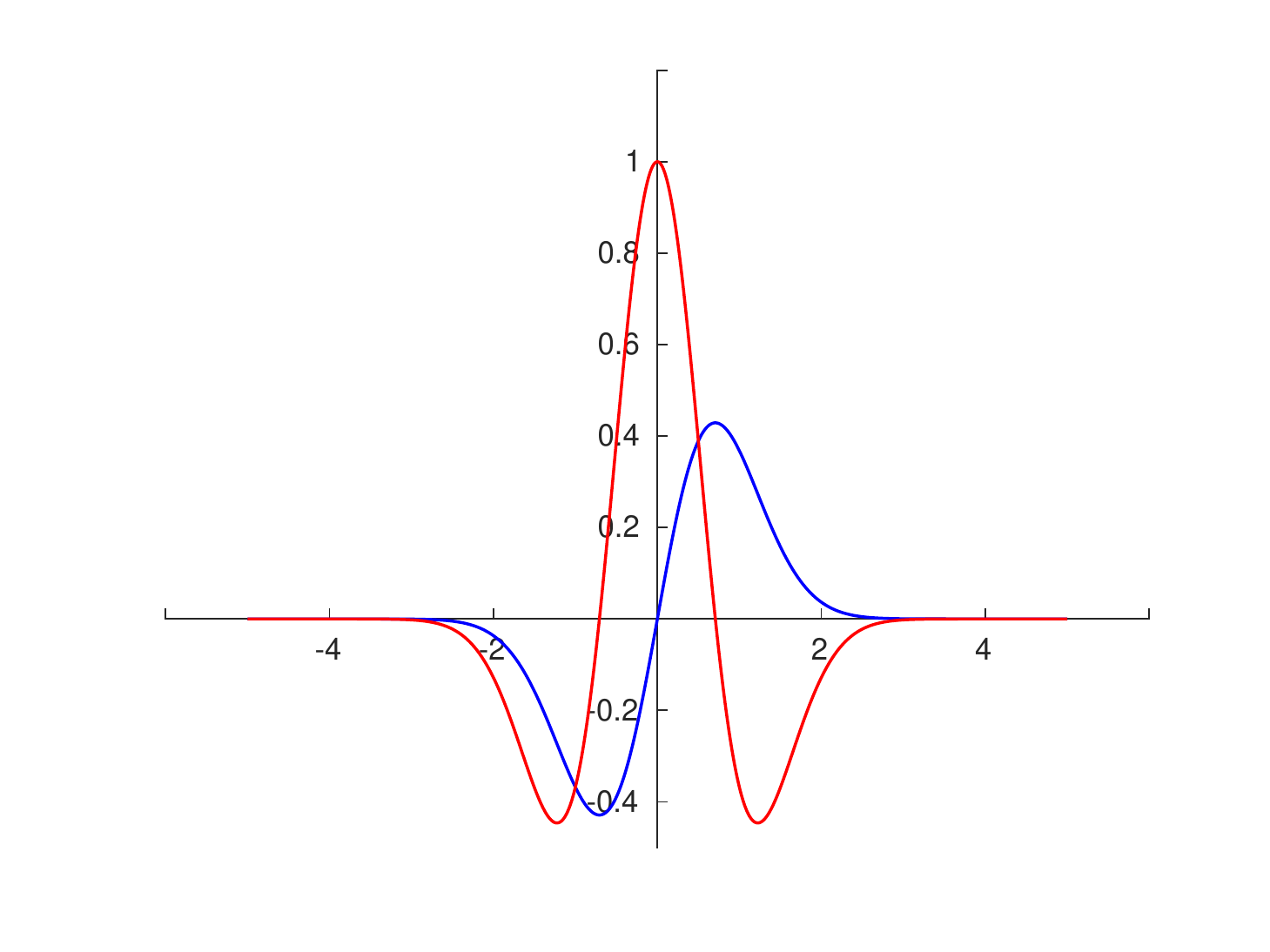}
			\protect\caption{The graph of $F : \mathbb{R} \rightarrow  \mathbb{R}, F(x)=xe^{-x^2}$, is in blue and the graph of $\nabla F : \mathbb{R} \rightarrow  \mathbb{R}, \nabla F(x)=(1-2x^2)e^{-x^2}$, is in red.}\label{Fig1}
			\par\end{centering}
	\end{figure}
\end{Example}

For the important particular case of strongly pseudo-monotone operators we will show
exponential convergence of the trajectories to the unique solution of VI($F, C$). To this end we need the following lemma.

\begin{Lemma}\label{ErrorBound}
Assume that $F$ is $\gamma$-strongly pseudo-monotone on $C$ with $\gamma >0$ and Lipschitz continuous with constant $L>0$.
Then for every $t\in [0,+\infty)$ we have
$$
\|x(t)-x^*\| \leq \frac{ 1+\lambda L+\lambda \gamma}{\lambda \gamma}\|x(t)-y(t)\|.
$$
\end{Lemma}

\begin{proof}
Let $x^* \in C$ be the unique solution to  {\rm VI($F, C$)} (see, for instance, \cite{Kim}) and $t \in [0,+\infty)$ fixed. Since $y(t) \in C$ we have 
$$
\left\langle F(x^*), y(t)-x^*\right\rangle \geq 0,
$$
which implies, according to the strong pseudo-monotonicity of $F$ on $C$, that
\begin{equation*} 
\left\langle F(y(t)), y(t)-x^*\right\rangle \geq \gamma \|y(t)-x^*\|^2.
\end{equation*}
Using the Lipschitz continuity of $F$ we get
\begin{align*}
\left\langle F(x(t)), x^*-y(t)\right\rangle 
&= \left\langle F(x(t))-F(y(t)), x^*-y(t) \right\rangle - \left\langle F(y(t)), y(t)-x^*\right\rangle\\
& \leq \|F(x(t))-F(y(t))\| \|y(t)-x^*\| -\gamma \|y(t)-x^*\|^2\\
& \leq L \|x(t)-y(t)\| \|y(t)-x^*\| -\gamma \|y(t)-x^*\|^2,
\end{align*}
which, in combination with \cref{DS2}, gives
\begin{equation*}
\begin{aligned}
\left\langle x^*-y(t), x(t)-y(t) \right\rangle 
& \leq \lambda \left\langle F(x(t)), x^*-y(t)\right\rangle \\
&\leq \lambda L \|x(t)-y(t)\| \|y(t)-x^*\| -\lambda \gamma \|y(t)-x^*\|^2
\end{aligned}
\end{equation*}
and, further,
\begin{equation*}
\begin{aligned}
\lambda \gamma \|y(t)-x^*\|^2
&\leq \lambda L \|x(t)-y(t)\| \|y(t)-x^*\|-\left\langle x^*-y(t), x(t)-y(t) \right\rangle \\
&\leq \lambda L \|x(t)-y(t)\| \|y(t)-x^*\|+\|x^*-y(t)\| \|x(t)-y(t)\|\\
&= \left( 1+\lambda L\right)  \|x(t)-y(t)\| \|y(t)-x^*\|.
\end{aligned}
\end{equation*}
This implies
$$
\|y(t)-x^*\| \leq \frac{ 1+\lambda L}{\lambda \gamma}\|x(t)-y(t)\|
$$
and, further,
\begin{equation}\label{Errorbound}
\|x(t)-x^*\| \leq \|x(t)-y(t)\|+\|y(t)-x^*\| \leq  \frac{ 1+\lambda L+\lambda \gamma}{\lambda \gamma}\|x(t)-y(t)\|.
\end{equation}
\end{proof}

\begin{Theorem}\label{Main result 2DS} 
	Assume that $F$ is $\gamma$-strongly pseudo-monotone on $C$ with $\gamma >0$ and Lipschitz continuous with constant $L>0$, and $0 < \lambda < \frac{1}{L}$.
	Then for every $t\in [0,+\infty)$ we have
	\begin{equation}\label{Exponen}
	\|x(t)-x^*\|^2 \leq \|x(0)-x^*\|^2\exp(-\alpha t),
	\end{equation}
	where $\alpha=:2(1-\lambda L)\left(\frac{\lambda \gamma}{1+\lambda L+\lambda \gamma}\right)^2 $ and $x^*$ is the unique solution of {\rm VI($F, C$)}.
\end{Theorem}

\begin{proof}
From Lemma \ref{ErrorBound} we have that for every $t \in [0,+\infty)$
	\begin{equation*}
	\|x(t)-x^*\| \leq  \frac{ 1+\lambda L+\lambda \gamma}{\lambda \gamma}\|x(t)-y(t)\|,
	\end{equation*}
	which, in combination with Proposition \ref{DSProperty}, leads to
	\begin{eqnarray*}
		\frac{1}{2}\frac{d}{dt} \|x(t)-x^*\|^2
		&=&\left\langle x(t)-x^*, \dot{x}(t) \right\rangle \\
		&\leq & -\left( 1-\lambda L\right) \|x(t)-y(t)\|^2 \\
		&\leq & -(1-\lambda L)\left(\frac{\lambda \gamma}{1+\lambda L+\lambda \gamma}\right)^2 \|x(t)-x^*\|^2.
	\end{eqnarray*}
Relation \eqref{Exponen} is a direct consequence of  Gronwall's Lemma. 
\end{proof}

\begin{Example} \label{Exam3}\rm
 If $M : H \rightarrow H$ is such that
	\begin{equation*}
	\langle Mx, x\rangle\geq \gamma\|x\|^2 \quad \forall x\in H,
	\end{equation*}
	for some $\gamma > 0$,  then one can show that the operator $F : H \rightarrow H$ in Example \ref{Exam1} is $\alpha \gamma$-strongly pseudo-monotone on $H$. On the other hand,
	$F$ is in general not monotone, as one can see in Figure \ref{Fig2} for a particular operator.
	\begin{figure}[!htbp]
		\begin{centering}
			\includegraphics[width=1\textwidth]{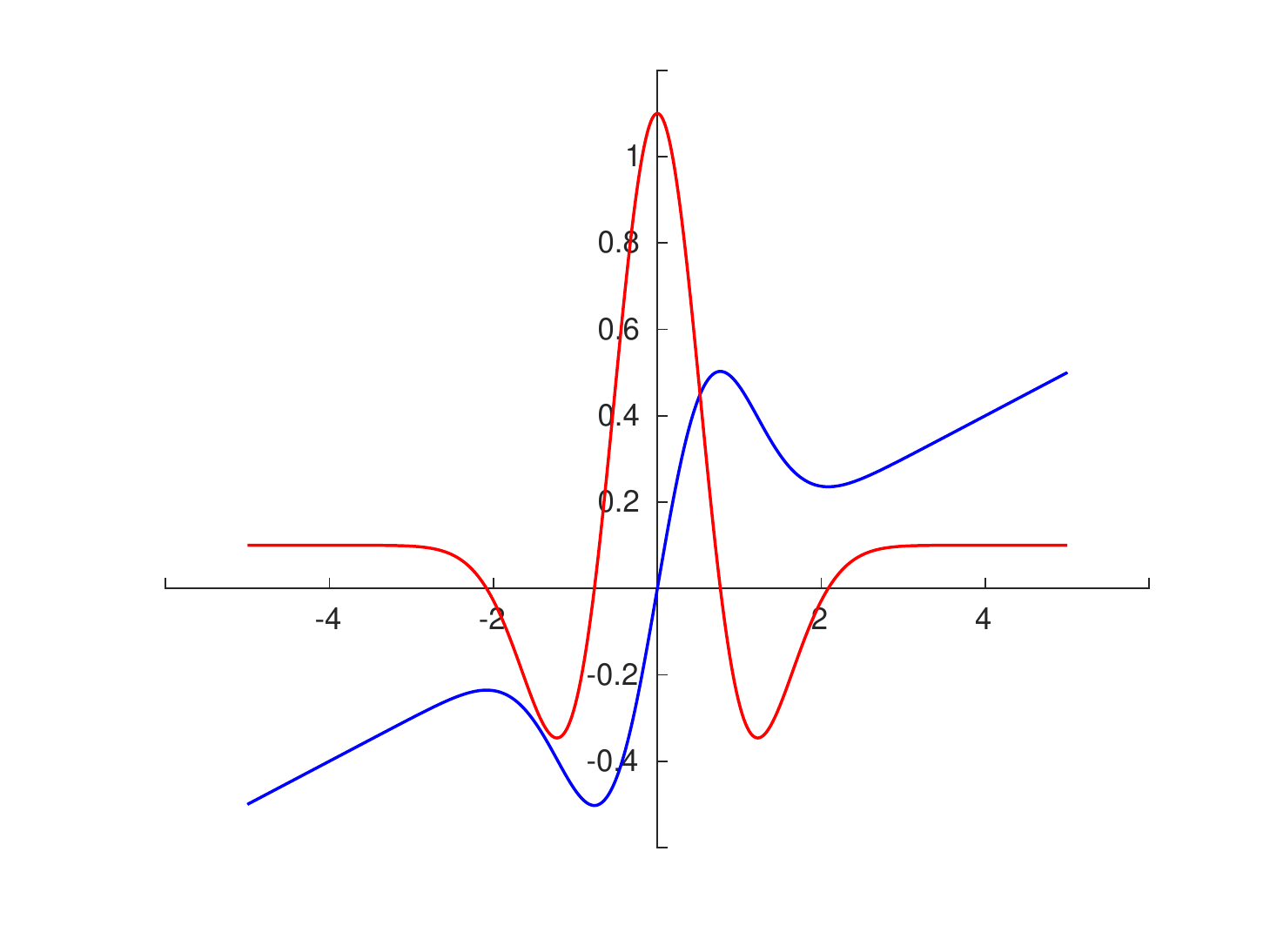}
			\protect\caption{The graph of $F : \mathbb{R} \rightarrow \mathbb{R}, F(x)=xe^{-x^2}+0.1x$, is in blue and the graph of $\nabla F : \mathbb{R} \rightarrow \mathbb{R}, \nabla F(x)=(1-2x^2)e^{-x^2}+0.1$, is in red.}\label{Fig2}
			\par\end{centering}
	\end{figure}
\end{Example}

\begin{Example} \label{Exam4}\rm
	Let $C= \{x \in [-5,5]^3 : x_1+x_2+x_3=0 \} \subseteq \mathbb{R}^3$ and $F: \mathbb{R}^3 \rightarrow \mathbb{R}^3$  be defined as
	\begin{equation*}
	F(x) =\left( e^{-\|x\|^2}+q\right) Mx,
	\end{equation*}
	where $q=0.2$ and
	$$
	M=\left[ {\begin{array}{ccc}
		1 & 0 & -1\\
		0 & 1.5&0 \\
		-1 &0 &2
		\end{array} } \right].
	$$
 As mentioned in Example \ref{Exam3}, $F$ is $\gamma$-strongly pseudo-monotone on $\mathbb{R}^3$ with 
	constant $\gamma:=q \cdot \lambda_{min} \approx 0.0764$, where $\lambda_{min}$ is the smallest eigenvalue of $M$, and Lipschitz continuous with constant $L \approx 5.0679$.
	Since for $x=(-1, 0, 0)^T, y=(-2, 0, 0)^T \in \mathbb{R}^3$
	$$
	\left\langle F(x)-F(y),x-y\right\rangle = -0.1312 <0,
	$$	
	$F$ is not monotone.	
		\begin{figure}[!htbp]
		\begin{centering}
			\includegraphics[width=1\textwidth]{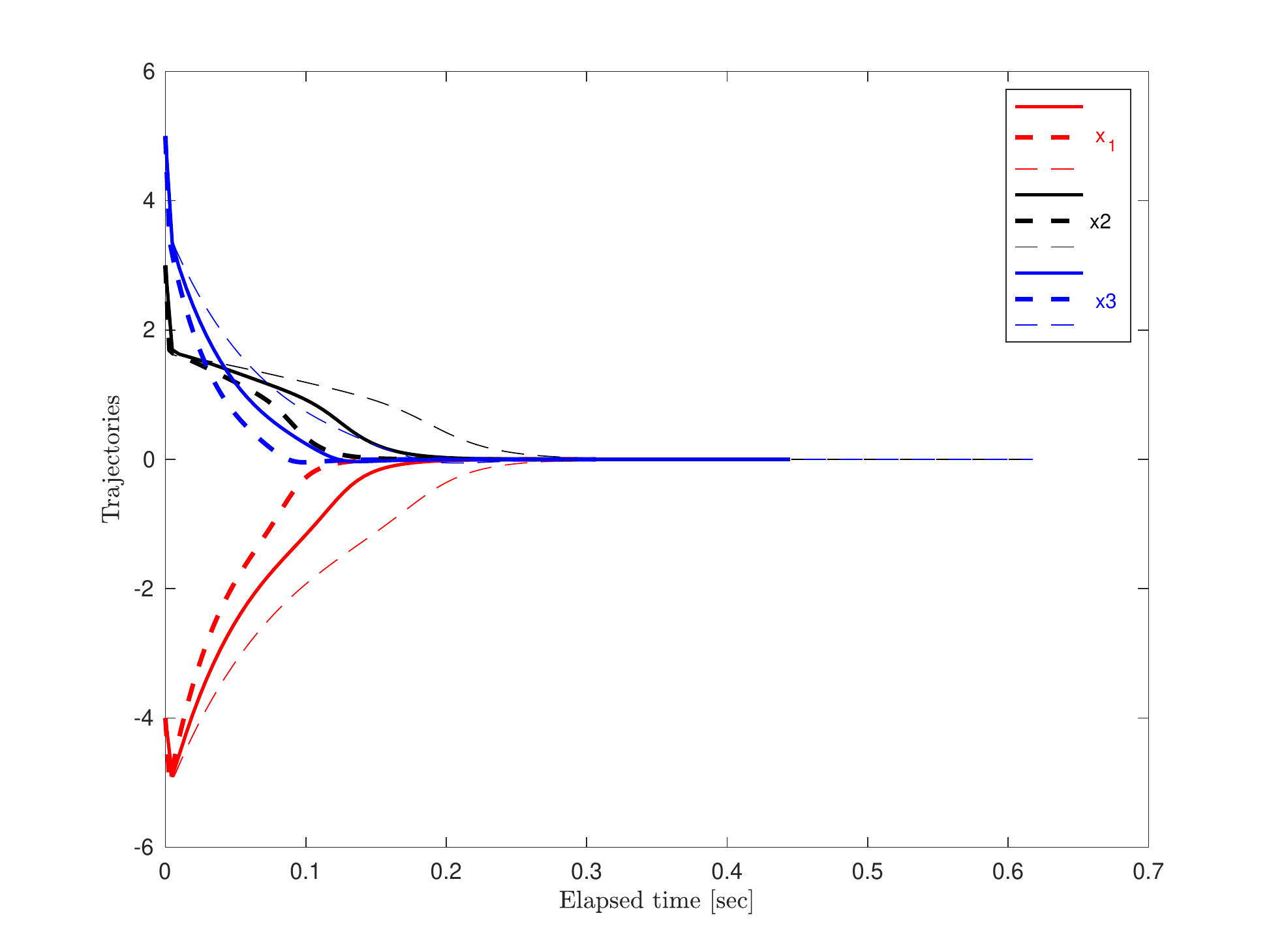}
			\protect\caption{Trajectories generated by the dynamical system \eqref{DS} for $x_0=(-4,3,5)^T$ and $\lambda= 0.99/L$ (continuous lines), $\lambda = 0.8/L$ (dashed thick lines), and $\lambda = 0.5/L$ (dashed thin lines).}
			\par\end{centering}\label{Fig3}
	\end{figure}

Figure \ref{Fig3} displays the trajectories generated by the dynamical system \eqref{DS} attached to VI($F,C$), with starting point $x_0=(-4,3,5)^T$ and different values of $\lambda$. These are represented for $\lambda= 0.99/L$ by continuous lines, for  $\lambda = 0.8/L$ by dashed thick lines, and for $\lambda = 0.5/L$ by dashed thin lines. They all converge exponentially to the unique solution $x^*=(0,0,0)^T$ of VI($F,C$). One can clearly see that the choice of $\lambda$ influences the speed of convergence, namely, the smaller the values of $\lambda$, the worse the convergence of the trajectories become.
\end{Example}

%% % % % % % % % % % % % % % % % % % % % % % %
\section{The forward-backward-forward algorithm with relaxation parameters} 
\label{sec:TsengMethod}
%% % % % % % % % % % % % % % % % % % % % % % %
In this section we analyze the convergence of Tseng's forward-backward-forward algorithm with relaxation parameters derived in Remark \ref{remark21} by the time discretization of the dynamical system \eqref{DS} in the context of solving pseudo-monotone variational inequalities.

\begin{algo}\label{alg:Tseng}

\textbf{Initialization:} Choose the starting point $x_0 \in H$, the step size $\lambda > 0$, and the sequence of relaxation parameters $(\rho_n)_{n\ge0}$. Set $n=0$.

\medskip
\textbf{Step 1:} Compute $$y_n=P_C(x_n-\lambda F(x_n)).$$

If $y_n=x_n$ or $F(y_n)=0$, then STOP: $y_n$ is a solution.

\medskip
\textbf{Step 2:} Set 
$$x_{n+1}= \rho_n \left( y_n + \lambda (F(x_n)-F(y_n))\right) +(1-\rho_n)x_n,$$  

update $n$ to $n+1$ and go to \textbf{Step 1}.
\end{algo}

\begin{Remark}
	If $\rho_n =1$ for every $n\ge 0$, then Algorithm \ref{alg:Tseng} reduces to the classical forward-backward-forward method proposed by Tseng in \cite{Tseng2000}.
\end{Remark}

For the convergence analysis we assume that Algorithm \ref{alg:Tseng} does not terminate after a finite number of iterations. In other words, we assume that for every $n \geq 0$ it holds $x_n \neq y_n$ and $F(y_n) \neq 0$.

\begin{Proposition}\label{SequenceProperty} 
Assume that the solution set $\Omega$ is nonempty and $F$ is pseudo-monotone on $C$ and Lipschitz continuous with constant $L$. Let $t_n:= y_n + \lambda (F(x_n)-F(y_n))$ for every $n \ge 0$. 
Then for every solution $x^* \in \Omega$ and every $n \geq 0$ it holds 
%\begin{itemize} \label{Main_Inequality}
%\item [(i)] $\|t_{n}-x^*\|^2\leq \|x_n-x^*\|^2- \left(1-\lambda^2 L^2 \right) \|y_{n} - x_n\|^2$;
%\item [(ii)] $\|x_{n+1}-x^*\|^2\leq \|x_n-x^*\|^2- \left(1-\lambda^2 L^2 \right) \|y_{n} - x_n\|^2-\rho_n(1-\rho_n) \|t_{n} - x_n\|^2$.
%\end{itemize}
\begin{equation} \label{Main_Inequality}
	\|x_{n+1}-x^*\|^2\leq \|x_n-x^*\|^2- \rho_n \left(1-\lambda^2 L^2 \right) \|y_{n} - x_n\|^2-\rho_n(1-\rho_n) \|t_{n} - x_n\|^2.
\end{equation}
\end{Proposition}

\begin{proof}
Let $x^*$ be an arbitrary element in $\Omega$ and $n \geq 0$ be fixed. Then we have
$$
\left\langle F(x^*), y-x^*\right\rangle \geq 0 \quad \forall y \in C. 
$$
Substituting $y:=y_n \in C$ into this inequality it yields
$$
\left\langle F(x^*), y_n-x^*\right\rangle \geq 0.
$$
From the pseudo-monotonicity of $F$ on $C$ it follows
\be \label{eq:pseudomototone}
\left\langle F(y_n), y_n-x^*\right\rangle \geq 0.
\ee
Since $y_n=P_C(x_n-\lambda F(x_n))$, according to \cref{projectionProperty}, we get
$$
\left\langle y-y_n,  y_n - x_n+\lambda F(x_n) \right\rangle  \geq 0 \quad \forall y \in C,
$$
which yields
\be \label{projectionOfyn}
\left\langle x^*-y_n,  y_n - x_n+\lambda F(x_n) \right\rangle  \geq 0.
\ee
Multiplying both sides of \cref{eq:pseudomototone} by $\lambda>0$ and adding the resulting inequality to \cref{projectionOfyn}, it yields
$$
\left\langle x^*-y_n,  y_n - x_n+\lambda F(x_n)-\lambda F(y_n) \right\rangle \geq 0
$$
or, equivalently,
$$
\left\langle x^*-y_n,  t_n - x_n \right\rangle \geq 0.
$$
This implies that
\begin{align}\label{Es1}
\left\langle  t_n- x^*,t_n - x_n \right\rangle 
& \leq  \left\langle  t_n- y_n,  t_n - x_n \right\rangle \nonumber \\
&=  \|t_n - x_n\|^2+ \left\langle  x_{n}- y_n,  t_n - x_n \right\rangle \nonumber\\
&=  \|t_n - x_n\|^2+ \left\langle  x_{n}- y_n,  y_n + \lambda (F(x_n)-F(y_n)) - x_n \right\rangle \nonumber \\
&=  \|t_n - x_n\|^2-\|y_{n} - x_n\|^2+  \lambda \left\langle  x_{n}- y_n,  F(x_n)-F(y_n)\right\rangle.
\end{align}
On the other hand, we have
\begin{equation} \label{Es2}
\|t_n - x^*\|^2-\|x_{n} - x^*\|^2+\|t_n - x_n\|^2=2\left\langle  t_n- x^*,  t_n - x_n \right\rangle. 
\end{equation}
Combining \eqref{Es1} and \eqref{Es2} we obtain
\begin{align}\label{Es3}
\|t_n - x^*\|^2 \leq & \ \|x_{n} - x^*\|^2+ \|t_n - x_n\|^2-2 \|y_{n} - x_n\|^2 \nonumber\\
& \ +2\lambda \left\langle  x_{n}- y_n,  F(x_n)-F(y_n)\right\rangle.
\end{align}
Using the Lipschitz continuity of $F$ we obtain
\begin{align}\label{eq:Es4}
\|t_n - x_n\|^2 &= \ \|y_n + \lambda (F(x_n)-F(y_n))-x_n\|^2 \nonumber \\
&= \ \|y_n - x_n\|^2+ 2\lambda \left\langle  y_{n}- x_n,  F(x_n)-F(y_n)\right\rangle+\lambda^2 \|F(x_n)-F(y_n)\|^2 \nonumber \\ 
&\leq \  \|y_n - x_n\|^2+ 2\lambda \left\langle  y_{n}- x_n,  F(x_n)-F(y_n)\right\rangle+ \lambda^2 L^2\|x_n-y_n\|^2.
\end{align}
Finally, from \eqref{Es3} and \eqref{eq:Es4} it yields
\begin{align*}
\|t_n - x^*\|^2 \leq  \|x_{n} - x^*\|^2-\left(1- \lambda^2 L^2\right)  \|y_{n} - x_n\|^2. 
\end{align*}
Moreover,
\begin{align*}
\|x_{n+1} - x^*\|^2 &=  \|\rho_n (t_n-x^*)+ (1-\rho_n) \left( x_{n} - x^*\right) \|^2\\
&=  \rho_n \|t_n-x^*\|^2+ (1-\rho_n) \| x_{n} - x^* \|^2 - \rho_n(1-\rho_n) \| t_{n} - x_n \|^2.
\end{align*}
By plugging this equality in the inequality above, we obtain the desired result.
\end{proof}

\begin{Remark}\rm
One can notice that the pseudo-monotonicity of $F$ was used in the proof of Proposition \ref{SequenceProperty} in order to obtain relation \eqref{eq:pseudomototone}. 
This means that the pseudo-monotonicity of $F$ can be actually replaced by the following weaker assumption (see \cite{DangLan, SolodovSvaiter99})
\begin{equation*}\label{relaxedpseudo}
\left\langle F(x), x-x^*\right\rangle \geq 0 \quad \forall x \in C \ \forall x^* \in \Omega,
\end{equation*}
which basically requires that every solution of the variational inequality of Stampacchia type is a solution of the variational inequality of Minty type.
\end{Remark}

\begin{Remark} \label{ReBounded}\rm
In contrast to the extragradient method, the sequence $(x_n)_{n \geq 0}$ generated by Algorithm \ref{alg:Tseng} may not be feasible. This is why we need to ask in the convergence analysis that $F$ is Lipschitz continuous on the whole space $H$. 
However, if the feasible set $C$ is bounded, then we can weaken this assumption by asking that $F$ is Lipschitz continuous on the bounded set 
	$$
	D:=\left\lbrace x +y: x \in C,\, \|y\| \leq d \right\rbrace,
	$$
where $d$ denotes the diameter of $C$. Notice that $C \subseteq D$. In this case, if we start Algorithm \ref{alg:Tseng} with an element $x_0 \in C$ and choose $0 < \lambda < \frac{1}{L}$, then from \eqref{Main_Inequality} and $\rho_0 \in [0,1]$ we have
	$$
	\|x_{1} - x^*\|^2 \leq  \|x_{0} - x^*\|^2-\rho_0 \left(1- \lambda^2 L^2\right)  \|y_{0} - x_0\|^2,
	$$
	which implies that $\|x_{1} - x^*\| \leq \|x_{0} - x^*\| \leq d $.  
	Since $x_1=x^*+x_1-x^*$, we have 	$x_1 \in D$. By induction, we obtain $\|x_{n} - x^*\| \leq d$ and therefore $x_n \in D$ for every $n \geq 0$. 
\end{Remark}

The following theorem states the convergence of the underrelaxed Tseng's method for pseudo-monotone variational inequalities.

\begin{Theorem}\label{Main result} 
Assume that the solution set $\Omega$ \!is nonempty, \!$F$ is pseudo-monotone on $H$, Lipschitz continuous with constant $L>0$ and sequentially weak-to-weak continuous, and  $0 < \lambda < \frac{1}{L}$. Assume also that $(\rho_n)_{n\ge0} \subseteq [0,1]$ and $\lim \inf_{n \to +\infty} \rho_n > 0$. Then the sequence  $(x_n)_{n \geq 0}$ generated by Algorithm \ref{alg:Tseng}  converges  weakly to a solution of {\rm VI($F, C$)}.
\end{Theorem}

\begin{proof}
Let $x^* \in \Omega$ be fixed. Since $(\rho_n)_{n\ge0} \subseteq [0,1]$ and $0 < \lambda < \frac{1}{L}$,  \cref{Main_Inequality} yields that the sequence $\left( \|x_{n} - x^*\|^2\right)_{n \geq 0}$ is monotonically decreasing and therefore convergent. 

To obtain the convergence of the sequence $(x_n)_{n \geq 0}$ to an element in $\Omega$ we only need to prove that every weak sequential cluster point of the sequence belongs to $\Omega$. The conclusion will follow from the Opial Lemma (see \cite[Theorem 5.5]{Bauschkebook}).

Relation \cref{Main_Inequality} also implies that
\begin{equation*} 
\lim_{n \to +\infty} \rho_n (1-\lambda^2 L^2)\|y_n-x_n\|=0,
\end{equation*}
which, since $\lim \inf_{n \to +\infty} \rho_n > 0$, further leads to
\begin{equation*}
\lim_{n \to +\infty} \|y_n-x_n\|=0.
\end{equation*}
Since $F$ is Lipschitz continuous on $H$, we have
$$
\|F(x_n)-F(y_n)\| \leq L \|x_n-y_n\| \quad \forall n \geq 0,
$$ 
hence,
$$
\lim_{n \to +\infty} \|F(x_n)-F(y_n)\|=0.
$$
Further, we consider $\hat{x}$, a weak sequential cluster point of $(x_n)_{n \geq 0}$, and a subsequence $(x_{n_k})_{k \geq 0}$ of $(x_n)_{n \geq 0}$ which converges weakly to $\hat{x}$ as $k \rightarrow +\infty$. Since $\lim_{k \to +\infty}\|x_{n_k}-y_{n_k}\|=0$, $(y_{n_k})_{k \geq 0}$ also converges weakly to $\hat{x}$ as $k \rightarrow +\infty$. 

We are now in the same situation as in the proof of Theorem \ref{Main result DS}, the role of the sequences $(x_n)_{n \geq 0}$ and $(y_n)_{n \geq 0}$ being played by  $(x_{n_k})_{k \geq 0}$ and $(y_{n_k})_{k \geq 0}$, respectively. Thus, arguing as in the proof of this theorem, we obtain that $\hat{x}\in \Omega$.
\end{proof}

\begin{Remark} \label{RemSteps}\rm
The conclusion of Theorem \ref{Main result} remains valid even if we replace in every iteration of Algorithm \ref{alg:Tseng} the fixed stepsize $\lambda >0$ by a variable stepsize $\lambda_n$, where the sequence $(\lambda_n)_{n \geq 0}$ fulfills
$$
0< \inf_{n \geq 0} \lambda_n \leq \sup_{n \geq 0} \lambda_n < \frac{1}{L}.
$$
On the other hand, when (an upper bound of) the Lipschitz constant of $F$ is not available, we can use in Algorithm \ref{alg:Tseng} the following adaptive stepsize strategy
\begin{equation*}\label{adaptiveStep}
	\lambda_{n+1}:= 
	\begin{cases}
	\begin{aligned}
	&\min \left\lbrace \frac{\mu \|x_{n}-y_n\|}{\|F(x_n)-F(y_n)\|}, \lambda_n \right\rbrace, \quad  \text{if} \ F(x_n)-F(y_n) \ne 0,\\
	 &\lambda_n, \quad \quad \quad  \text{otherwise},
	\end{aligned}
	\end{cases}
\end{equation*} 
where $\mu \in (0,1) $ and $\lambda_0 > 0$. The sequence $(\lambda_n)_{n \geq 0}$ is nonincreasing.  If, for $n \geq 0$, $F(x_n)-F(y_n) \ne 0$, then it holds
$$
\frac {\mu \|x_{n}-y_n\|}{\|F(x_n)-F(y_n)\|} \geq \frac {\mu \|x_{n}-y_n\|}{L \|x_n-y_n\|}=\frac{\mu}{L},
$$
which shows that  $(\lambda_n)_{n \geq 0}$  is bounded from below by $\min\left\lbrace \lambda_0, \frac{\mu}{L}\right\rbrace > 0$. Notice that, if $\lambda_0 \leq \frac{\mu}{L}$, then $(\lambda_n)_{n \geq 0}$ is a constant sequence,
which leads to a fixed stepsize strategy. Consequently, the limit $\lim_{n \to +\infty} \lambda_n$  exists and it is a positive real number. 

This means that we can adapt the proof of Proposition \ref{SequenceProperty} to the new adaptive stepsize strategy and, by taking into consideration \eqref{eq:Es4}, we get instead of
\eqref{Main_Inequality}
\begin{equation*}\label{Main_Inequality2}
\|x_{n+1}-x^*\|^2\leq \|x_n-x^*\|^2- \rho_n \left(1-\frac{\lambda_n^2 \mu^2}{\lambda_{n+1}^2}  \right) \|y_{n} - x_n\|^2 \ \forall n \geq 0.
\end{equation*}
Due to $\lim \inf_{n \to +\infty} \rho_n > 0$ and $\lim_{n \to +\infty } \left( 1-\frac{\lambda_n^2 \mu^2}{\lambda_{n+1}^2}\right)   = 1-\mu^2 > 0$, there exists $N>0$ such that 
$$
\|x_{n+1}-x^*\| \leq \|x_n-x^*\| \quad \forall n \geq N, 
$$
which implies that $\lim_{n \to +\infty }\|x_n-x^*\|$ exists and $\lim_{n \to +\infty }\|x_n-y_n\|=0$. From here, one can carry out the same convergence analysis as for the fixed stepsize strategy.
\end{Remark}

\begin{Remark}\rm
If the operator $F$ is monotone on $C$, then it is not necessary to impose that $F$ is sequentially weak-to-weak continuous. Indeed, for $y \in C$ fixed, we obtain for the subsequences $(x_{n_k})_{k \geq 0}$ and $(y_{n_k})_{k \geq 0}$ arising in the proof of Theorem \ref{Main result} (see also \eqref{ine-1} and the proof of Theorem \ref{Main result DS})
 \begin{eqnarray*}
\frac{1}{\lambda}\left\langle x_{n_k}-y_{n_k}, y-y_{n_k} \right\rangle 
 &\leq&   \left\langle  F(x_{n_k})-F(y_{n_k}),y-y_{n_k}  \right \rangle 
+ \left\langle  F(y_{n_k}),y-y_{n_k}  \right \rangle \\
 &\leq&   \left\langle  F(x_{n_k})-F(y_{n_k}),y-y_{n_k}  \right \rangle 
 + \left\langle  F(y),y-y_{n_k}  \right \rangle \quad \forall k \geq 0. 
 \end{eqnarray*}
Letting $k \to +\infty$ we get
$$
\left\langle  F(y),y-\hat{x}  \right \rangle \geq 0
$$
and this leads to the desired conclusion.
\end{Remark}

In the following we will show that the convergence result in Theorem \ref{Main result} follows in finite dimensional spaces under weaker assumptions.

\begin{Theorem}
Let $H$ be a  finite dimensional real Hilbert space. Assume that the solution set $\Omega$ is nonempty, $F$ is pseudo-monotone on $C$ and Lipschitz continuous with constant $L>0$, and  $0 < \lambda < \frac{1}{L}$.  
Assume also that $(\rho_n)_{n\ge0} \subseteq [0,1]$ and $\lim \inf_{n \to +\infty} \rho_n > 0$. Then the sequence  $(x_n)_{n \geq 0}$ generated by Algorithm \ref{alg:Tseng}  converges  to a solution of {\rm VI($F, C$)}.
\end{Theorem}

\begin{proof}
Let $x^* \in \Omega$ be fixed. Since $0 < \lambda < \frac{1}{L}$,  from \cref{Main_Inequality} it follows that the sequence 
$\left( \|x_{n} - x^*\|^2\right)_{n \geq 0}$ is monotonically decreasing and therefore convergent. In addition we have
$$
\lim_{n \to +\infty} \|y_n-x_n\|=0.
$$
As $(x_n)_{n \geq 0}$ is bounded, there exists a subsequence $(x_{n_k})_{k \geq 0}$ of it, which converges to an element $\hat{x}$ as $k \rightarrow +\infty$. 
Since $\lim_{n \to +\infty}\|x_{n_k}-y_{n_k}\|=0$, $(y_{n_k})_{k \geq 0}$ also converges to $\hat{x}$ as $k \rightarrow +\infty$. 

Let now $y \in C$ be fixed. Then we have that
$$
\left\langle y-y_{n_k},  y_{n_k} - x_{n_k}+\lambda F(x_{n_k}) \right\rangle  \geq 0 \quad \forall k \geq 0. 
$$
Taking the limit as $k \to +\infty $ and using that $F$ is continuous, we obtain 
$$
\left\langle y-\hat{x}, F(\hat{x}) \right\rangle  \geq 0.
$$
Since $y \in C$ has been arbitrarily chosen, it follows that $\hat{x}$ is a solution of {\rm VI($F, C$)}. 

Replacing in \cref{Main_Inequality} $x^*$ with $\hat{x}$, it yields that the sequence $\left( \|x_{n} - \hat x\|\right)_{n \geq 0}$ is convergent. Since 
$ \lim_{k \to +\infty } \|x_{n_k} - \hat{x}\| =0 $, it follows that $\lim_{n \to +\infty } x_{n} = \hat{x}$.
\end{proof}

In the next theorem we show that one can consider even an overrelaxation of the forward-backward-forward algorithm without altering its convergence behaviour.

\begin{Theorem}\label{Main resultOverRelaxed} 
	Assume that the solution set $\Omega$ \!is nonempty, \!$F$ is pseudo-monotone on $H$, Lipschitz continuous with constant $L>0$ and sequentially weak-to-weak continuous, and $0 < \lambda < \frac{1}{L}$. Assume also that $(\rho_n)_{n\ge0} \subseteq [1,2)$ and $\limsup_{n \to +\infty} \rho_n < 2-\frac{2\lambda L}{1+\lambda L} $. Then the sequence  $(x_n)_{n \geq 0}$ generated by Algorithm \ref{alg:Tseng}  converges  weakly to a solution of {\rm VI($F, C$)}.
\end{Theorem}
\begin{proof}
In view of Proposition \ref{SequenceProperty} we have that
\begin{equation} \label{Main_InequalityRelax}
\|x_{n+1}-x^*\|^2\leq \|x_n-x^*\|^2- \rho_n \left(1-\lambda^2 L^2 \right) \|y_{n} - x_n\|^2 +\rho_n(\rho_n-1) \|t_{n} - x_n\|^2 \ \forall n \geq 0.
\end{equation}
By the Lipschitz continuity of $F$ we have for all $n \ge 0$ that 
$$
\|t_{n} - x_n\|^2 = \|y_{n} - x_n + \lambda \left( F(x_n) - F(y_n)\right) \|^2 
\le \left(1+\lambda L \right)^2 \|y_{n} - x_n\|^2.
$$
Therefore, from \eqref{Main_InequalityRelax} we obtain
\begin{eqnarray*} \label{Main_InequalityRelax2}
\|x_{n+1}-x^*\|^2
&\leq& \|x_n-x^*\|^2- \rho_n \left(1-\lambda^2 L^2 \right) \|y_{n} - x_n\|^2\\ \nonumber
& & +\rho_n(\rho_n-1) \left(1+\lambda L \right)^2\|y_{n} - x_n\|^2\\\nonumber
&=&\|x_n-x^*\|^2- \rho_n \left(1-\lambda^2 L^2 -(\rho_n-1) \left(1+\lambda L \right)^2\right) \|y_{n} - x_n\|^2.
\end{eqnarray*}
Since $(\rho_n)_{n\ge0} \subseteq [1,2)$ and $\limsup_{n \to +\infty} \rho_n < 2-\frac{2\lambda L}{1+\lambda L} $, it is easy to check that 
$$
\liminf_{n \to +\infty} \left(1-\lambda^2 L^2 -(\rho_n-1) \left(1+\lambda L \right)^2\right) > 0.
$$
Hence, there exists $N \geq 0$ such that the sequence 
$\left( \|x_{n} - x^*\|^2\right)_{n \geq N}$ is monotonically decreasing and therefore convergent. In addition,  we have
$$
\lim_{n \to +\infty} \|y_n-x_n\|=0.
$$
The rest of the proof can be done in analogy to the  proof of of Theorem \ref{Main result}, relying on the Opial Lemma and on arguments from the proof of Theorem \ref{Main result DS}.
\end{proof}

\begin{Example} \label{Exam1bis} \rm A differentiable function $f: E \rightarrow \mathbb{R}$, where $E \subseteq \mathbb{R}^n$ is an open set, is called \emph{pseudo-convex on $E$}, if for every $x,y \in E$ it holds
$$\langle \nabla f(x), y-x \rangle \geq 0 \quad \Rightarrow \quad f(y) \geq f(x).$$
It is well-known that $f$ is pseudo-convex on $E$ if and only if $\nabla f$ is pseudo-monotone on $E$ (\cite{KaramardianSchaible90}). Algorithm \ref{alg:Tseng} can be used to solve optimization problems of the form
	\begin{equation*} \label{optiProb}
	\min_{x \in C}\, f(x),
	\end{equation*}
where $f: \mathbb{R}^n \to \mathbb{R}$ is a differentiable function with Lipschitz continuous gradient which is also pseudo-convex on an open set $E \subseteq  \mathbb{R}^n$, and $C \subseteq E$ is a nonempty, convex and closed set. The class of pseudo-convex functions has been investigated in \cite{Mangasarian}, while characterizations of quadratic pseudo-convex functions have been provided in \cite{CottleFerland}.

A important subclass of the one of pseudo-convex functions are ratios of convex and concave functions. Indeed, if $E \subseteq \mathbb{R}^n$ is a convex set, $g:E \to [0,+\infty)$ is a convex function, $h: E \to (0,+\infty)$ is a concave function, and both $g$ and $h$ are differentiable on $E$, then the function
$$f : E \rightarrow [0,+\infty), \quad f(x) := \frac{g(x)}{h(x)},$$
is pseudo-convex on $E$ (\cite{BL06}).
\end{Example}

In the following we show that when $F$ is strongly pseudo-monotone on $C$, then Algorithm \ref{alg:Tseng} generates a sequence which converges linearly to the unique solution of  {\rm VI($F, C$)}. 
We extend in this way a result proved by Tseng in \cite{Tseng2000} for strongly monotone operators.

\begin{Theorem}\label{Main result2} 
Assume that $F$ is $\gamma$-strongly pseudo-monotone on $C$ with $\gamma >0$ and Lipschitz continuous with constant $L>0$, and $0< \lambda < \frac{1}{L}$. Assume also that $(\rho_n)_{n\ge0} \subseteq [0,1]$.  Let $x^*$ be the unique  solution of the problem {\rm VI($F, C$)}.  Then
$$ 
\| x_{n+1}-x^*\| \leq \delta_n \| x_{n}-x^*\| \quad \forall n \geq 0,
$$
where $\delta_n:=\left( 1-\rho_n \left( 1-\lambda^2 L^2\right)  \left( \frac{\lambda \gamma }{1+\lambda L+\lambda \gamma} \right)^2 \right)^{1/2} \in (0,1)$.\\
In addition, if 
$\lim \inf_{n \to +\infty} \rho_n > 0$, 
then the sequence  $(x_n)_{n \geq 0}$  converges linearly to $x^*$.
\end{Theorem}

\begin{proof}
Let $n \geq 0$ be fixed. As in the proof of Lemma \ref{ErrorBound} (see \eqref{Errorbound}), one can show that
\begin{equation}\label{Errorbound2}
\|x_n-x^*\| \leq \|x_n-y_n\|+\|y_n-x^*\| \leq  \frac{ 1+\lambda L+\lambda \gamma}{\lambda \gamma}\|x_n-y_n\|.
\end{equation}
From \cref{Errorbound2} and \cref{Main_Inequality} we obtain
$$
\|x_{n+1}-x^*\|^2 
\leq \left( 1- \rho_n \left( 1-\lambda^2 L^2\right)  \left( \frac{\lambda \gamma }{1+\lambda L+\lambda \gamma} \right)^2 \right)  \|x_n-x^*\|^2,
$$
therefore,
$$
\|x_{n+1}-x^*\| \leq \delta_n  \|x_n-x^*\|,
$$
where $\delta_n:=\left( 1-\rho_n \left( 1-\lambda^2 L^2\right)  \left( \frac{\lambda \gamma }{1+\lambda L+\lambda \gamma} \right)^2 \right)^{1/2} \in (0,1)$. 
Now,  if $\lim \inf_{n \to +\infty} \rho_n > 0$, then we have $\lim \sup_{n \to +\infty} \delta_n < 1$, which means that 
$(x_n)_{n \geq 0}$ converges linearly to $x^*$.
\end{proof}

One can prove in a similar way linear convergence for the sequence generated by the overrelaxed variant of the forward-backward-forward algorithm.

\begin{Theorem}\label{Main result2Relax} 
	Assume that $0 < \lambda < \frac{1}{L}$ and $F$ is $\gamma$-strongly pseudo-monotone on $C$ with $\gamma >0$ and Lipschitz continuous with constant $L>0$. Assume also that $(\rho_n)_{n\ge0} \subseteq [1,2)$ and $\lim \sup_{n \to +\infty} \rho_n < 2-\frac{2\lambda L}{1+\lambda L}$.  Then the sequence  $(x_n)_{n \geq 0}$  converges linearly to the unique  solution $x^*$ of {\rm VI($F, C$)}.
\end{Theorem}

%% % % % % % % % % % % % % % % % % % % % % % %
\section{Numerical experiments} 
\label{sec:Numerical}

In this section we present two numerical experiments which we carried out in order to compare Algorithm \ref{alg:Tseng} with other algorithms in the literature designed for solving
pseudo-monotone variational inequalities. We implemented the numerical codes in Matlab and performed all computations on a Linux desktop with an Intel(R) Core(TM) i5-4670S processor at
3.10GHz. In our experiments we considered only variational inequalities governed by pseudo-monotone operators which are not monotone.

Let be VI($F,C$) with
$$
C = \left \{x \in \mathbb{R}^5: \sum_{i=1}^{m} x_i  \leq 5, 0 \leq x_i \leq 5 \right\}
$$
and
$$F : \mathbb{R}^5 \rightarrow \mathbb{R}^5,  F(x) = \left( e^{-\|x\|^2}+\alpha \right) \left( Mx+p\right),$$
where $\|\cdot\|$ denotes the Euclidean norm on $\mathbb{R}^5$, $\alpha=0.1$, $p=(-1, 2, 1, 0,-1)^T$ and 
\begin{equation}\label{M}
M:=
\left[ {\begin{array}{ccccc}
	5 & -1 &2&0&2\\
	-1 & 6&-1&3&0 \\
	2 &-1 &3&0&1\\
	0 & 3& 0 & 5&0\\
	2 & 0&1&0&4
	\end{array} } \right]
\end{equation}
is a positive definite matrix. We computed the unique solution $x^*$ of the variational inequality VI($F,C$) by running $10000$ iterations of Algorithm \ref{alg:Tseng} with $\rho_n=1$ for all $n \geq 0$ and stepsize $\lambda=\frac{0.5}{L}$. 

In a first experiment, we considered different variants of Algorithm \ref{alg:Tseng} with constant relaxation parameters $\rho_n = \rho$ for all $n \geq 0$, chosen such that $\rho  < 2 - \frac{2 \lambda L}{1+\lambda L}=\frac{4}{3}$.  The aim was to see to which extend the relaxation parameter does influence the convergence behaviour of the method. We considered $x_0=(1,3,2,1,4)^T$ as starting point and $\|x_n-x^*\| \leq 10^{-6}$ as stopping criterion. The projection on the set  $C$ was computed by using the {\tt quadprog} function in Matlab. 

In Table \ref{rFBF} we present the performances of the algorithm for different values of the relaxation parameter. It can be seen that the larger the values of the relaxation parameter $\rho$, the better the algorithm performs. This shows how important is it to investigate overrelaxed algorithms from both theoretical and numerical perspective.

\begin{table}[ht!]\centering
	\protect\caption{Comparison of the performances of Algorithm \ref{alg:Tseng} for different values of $\rho$.}\label{rFBF}
		\begin{tabular}{|c|c|c|c|c|c|c|c|c|c|}
		\hline 
		$\rho$ &$0.5$& $0.6$& $0.7$& $0.8$& $0.9$& $1$&$1.1$& $1.2$& $1.3$\\
		\hline
		Iterations & 236 & 195&166& 144& 127&112&90&93&88\\		 
		\hline
		Time (sec) & 2.06   & 1.10 &0.89 &0.74&0.60&0.55&0.47&0.56&0.44\\	 
		\hline				
	\end{tabular}
\end{table}

\begin{figure}[!ht]
\begin{centering}
\includegraphics[width=1\textwidth]{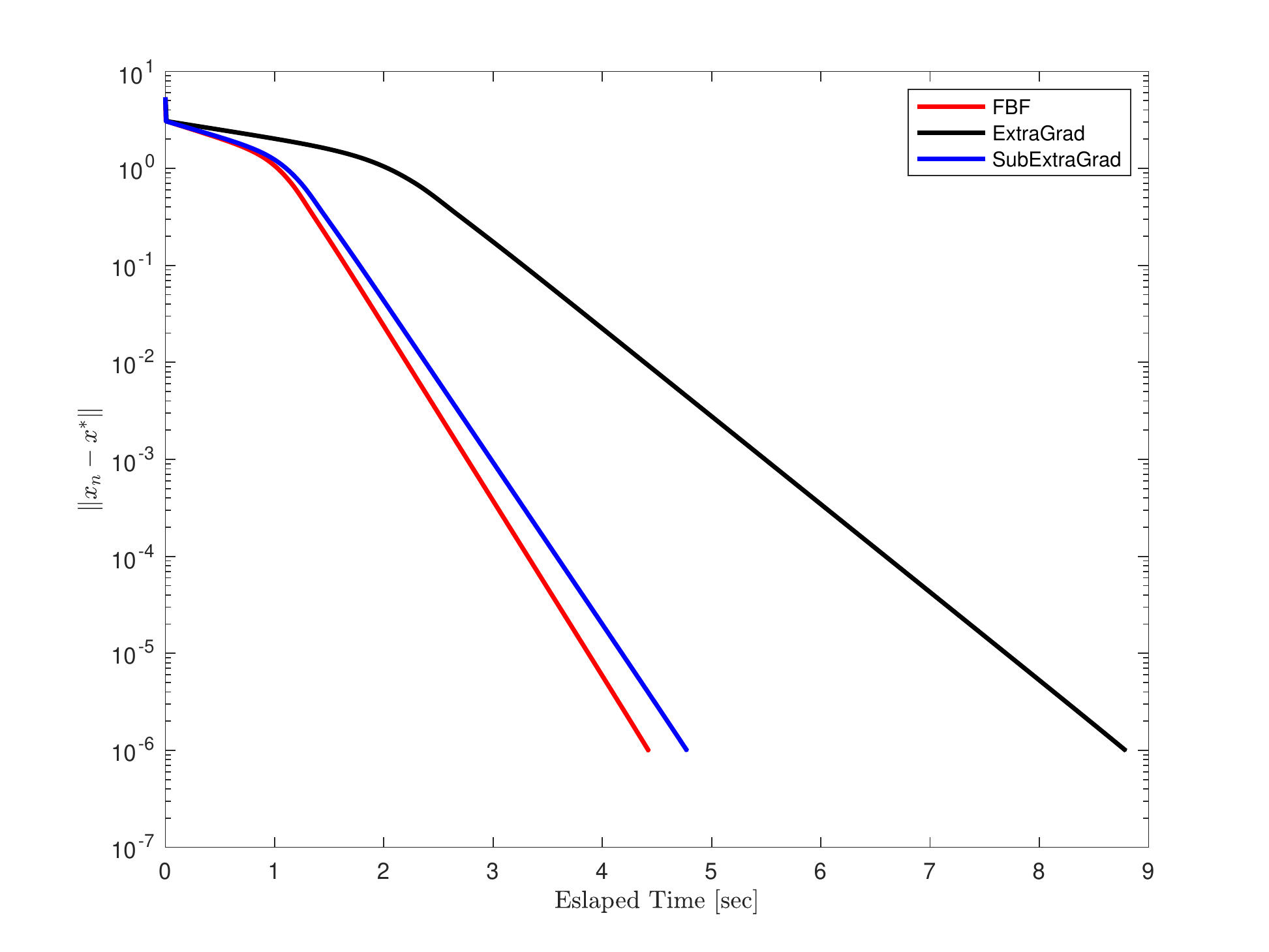}
\protect\caption{Comparison of the convergence behaviour of the forward-backward-forward method without relaxation (FBF), the extragradient method (ExtraGrad), and the  subgradient-extragradient method (SubExtraGrad) with stepsize $\lambda=0.99/L$.}\label{fig4}
\par\end{centering}
\end{figure}

In a second experiment we compared for the same problem the performances of the forward-backward-forward method without relaxation, the extragradient method and the subgradient-extragradient method, by considering for all three methods as stepsize $\lambda=0.99/L$. It can be seen in Figure \ref{fig4} that the forward-backward-forward method outperforms the extragradient method, being at least two times faster. This is not surprising, since the extragradient method requires two projections on the set $C$ at each iteration, while the forward-backward-forward method requires only one. It can be also notice that the latter also slightly outperforms the subgradient-extragradient method.
 \begin{figure}[!ht]
	\begin{centering}
		\includegraphics[width=1\textwidth]{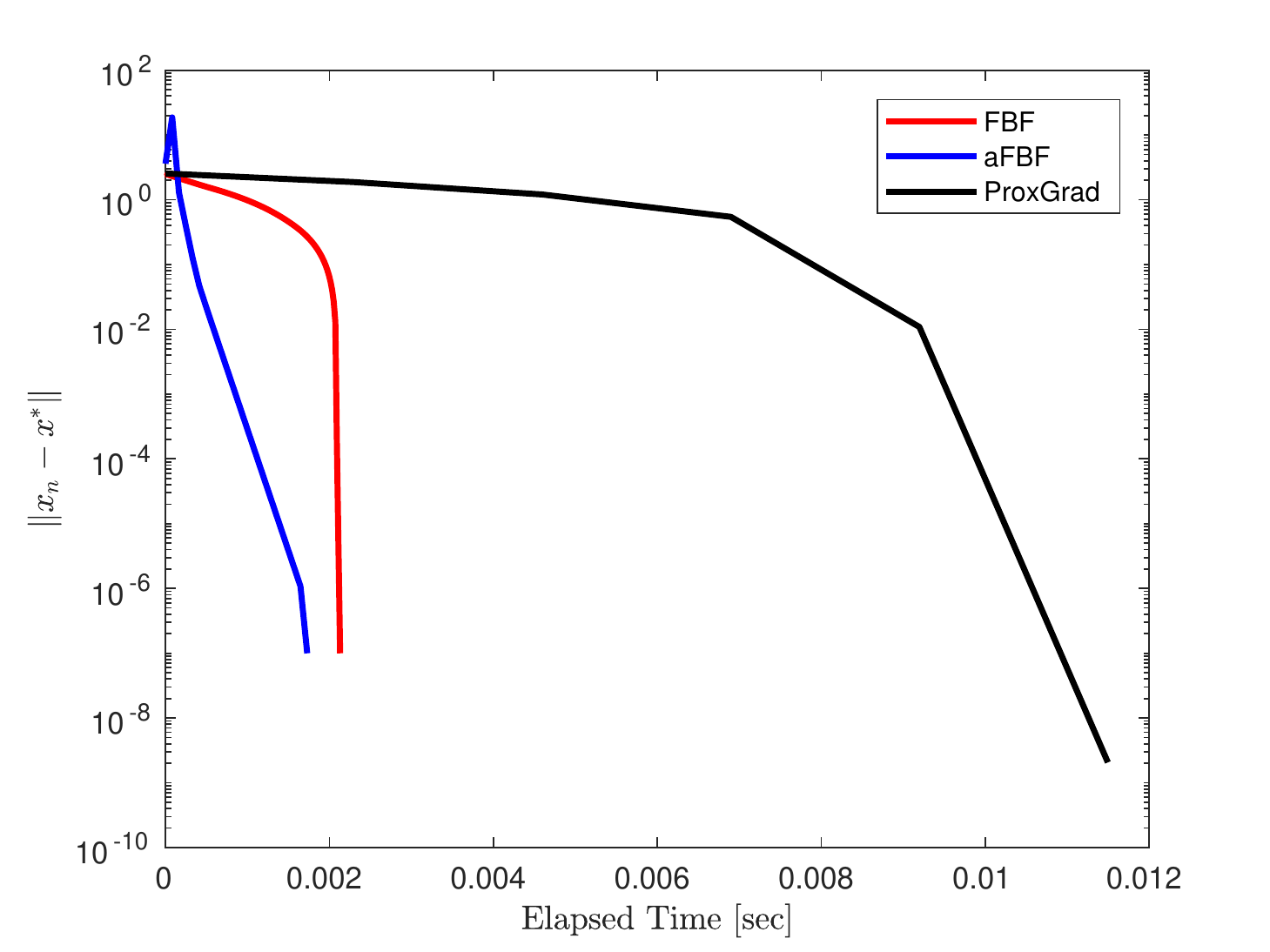}
		\protect\caption{Comparison of the convergence behaviour of the forward-backward-forward method (FBF) with fixed stepsize, the one with adaptive stepsize (aFBF)  and the proximal-gradient method (ProxGrad) for the fractional programming problem \eqref{quadpro}.}\label{FracProg}
\end{centering}
\end{figure}

In a third experiment we considered the quadratic fractional programming problem
 \begin{equation} \label{quadpro}
 \min_{x \in C} f(x):=\frac{x^TMx+a^Tx+c}{b^Tx+d},
 \end{equation}
with
 $$
C = \{x \in \mathbb{R}^5: 1 \leq x_i \leq 3, \ i=1,2, 3,4, 5\},
 $$
$M$ taken as in \eqref{M}, $a=(1,2,-1,-2,1)^T$, $b=(1,0,-1,0,1)^T$, $c=-2$ and $d=20$. According to the discussion in Example \ref{Exam1bis}, $f$ is pseudo-convex on the open set
$E:=\{x \in \mathbb{R}^5: b^Tx+d = x_1-x_3 +x_5 + 20 >0\}$, which implies that
$$
F: \mathbb{R}^5 \rightarrow \mathbb{R}^5,  \ F(x)=\nabla f(x):=\frac{\left( b^Tx+d\right) \left( 2Mx+a\right)-b\left( x^TMx+a^Tx+c\right)  }{\left( b^Tx+d\right)^2 },
$$
is pseudo-monotone on $E$. One can also notice that $C \subseteq E$.

In order to show the Lipschitz continuity of $F$, since $C$ is bounded, according to Remark \ref{ReBounded} it is enough to prove that this property holds on the set
\begin{align*}
 D & \ = \{x +y \in \mathbb{R}^5: x \in C, \|y\| \leq 2\sqrt{5}\}\\ 
 & \ = \{x \in \mathbb{R}^5: 1-2\sqrt{5} \leq x_i \leq 3+2\sqrt{5}, i=1,2, 3,4,5\}.
\end{align*}
Notice that $C \subseteq D \subseteq E$.

One can easily see that $\|\nabla F(x)\| \leq 148.68=:L >0$ for all $x \in D$, which means according to  the Mean Value Theorem that $F$ is Lipschitz continuous on $D$ with constant $L$. For this numerical experiment we assumed that the constant $L$ is not known in advance and used the adaptive stepsize strategy described in Remark \ref{RemSteps} with $\mu=0.9$ and $\lambda_0=1$. We compared the forward-backward-forward method (FBF) with fixed stepsize $\lambda=0.9/L$, with the one with adaptive stepsize (aFBF)  and the proximal-gradient (ProxGrad) method for fractional programming proposed in \cite[Algorithm 6]{BC17}.
We considered as starting point $x_0=(3,1.5,2,1.5,2)^T$ and as stopping criterion $\|x_n-x^*\| \leq 10^{-6}$.  The optimal solution of 
\eqref{quadpro} $x^*=(1,1,1,1,1)^T$ was obtained by running $10000$ iterations of Algorithm \ref{alg:Tseng} with $\rho_n=1$ for all $n \geq 0$. We solved the quadratic subproblem in \cite[Algorithm 6]{BC17} by using the \emph{quadprog} function in Matlab. The numerical performances of the three methods are 
displayed in Figure \ref{FracProg}. One can notice that the adaptive method aFBF is faster than FBF.  Moreover, both FBF and aFBF outperform the proximal-gradient method from \cite[Algorithm 6]{BC17}. A possible reason is that, while for the first two methods the projection on the set $C$ is computed explicitly, in every iteration of the proximal-gradient method a subproblem is solved by an external solver.

%% % % % % % % % % % % % % % % % % % % % % % %
\section{Conclusions and further research} 
\label{sec:Conclusions}

The object of our investigation was a variational inequality of Stampacchia type over a nonempty, convex and closed set governed by a pseudo-monotone and Lipschitz continuous operator.  We associated to it a forward-backward-forward dynamical system and carried out a Lyapunov-type analysis in order to prove the asymptotic convergence of the generated trajectories to a solution of the variational inequality. The explicit time discretization of the dynamical system leads to Tseng's 
forward-backward-forward algorithm with relaxation parameters. We proved convergence of the generated sequence of iterates to a  solution of the variational inequality as well as linear convergence rate under strong pseudo-monotonicity. Numerical experiments show that, when applied to pseudo-monotone variational inequalities over polyhedral sets, the overrelaxed variant algorithm has a better convergence behaviour when compared to other variants and also that Tseng's method outperforms Korpelevich's extragradient method and also the subgradient-extragradient method.

A topic of current interest is the formulation of numerical algorithms for minimax problems, due to its relevance for the training of generative adversarial networks (GANs). We want to investigate the convergence property of Tseng's forward-backward-forward method when solving the variational inequality to which the optimality conditions for the minimax problem give rise in both a deterministic and a stochastic setting and possibly to derive ergodic convergence rates for the gap function associated to the minimax problem.

{\bf Acknowledgements.} The authors are grateful to three anonymous reviewers for their pertinent comments and remarks which improved the quality of the paper.

%Finally, figure \ref{DSTrajectories} displays the trajectories generated by dynamical system \eqref{DS} for experiment 2 with $x_0=(-5,-3,3,2,4)^T$. It can be seen that the trajectories converges exponentially to the unique solution $$
%x^*=(1,1,1,1,1)^T.$$    
%
%\begin{figure}[!ht]
%	\begin{centering}
%		\includegraphics[width=1\textwidth]{Trajectories_m5.pdf}
%		\protect\caption{Trajectories generated by dynamical system \eqref{DS} for $x_0=(-5,-3,3,2,4)^T$ in experiment 2}\label{DSTrajectories}
%		\par\end{centering}
%\end{figure}
% % % % % % % % % % % % % % % % % % % % % % %

%% % % % % % % % % % % % % % % % % % % % % % %


\begin{thebibliography}{99}	
\bibitem{AAS2014}	{\sc Abbas, B., Attouch, H. and Svaiter, B. F.}: \emph{ Newton-like dynamics and forward-backward methods for structured monotone inclusions in Hilbert spaces}, J. Optim. Theory Appl. 161(2) (2014), pp. 331--360.

\bibitem{BB18} {\sc Banert, S. and Bo\c t, R. I.}: \emph{A forward-backward-forward differential equation and its asymptotic properties}, J. Conv. Anal. 25(2) (2018), pp. 371--388.

\bibitem{Bauschkebook} {\sc Bauschke, H. H. and Combettes, P. L.}: \emph{Convex Analysis and
Monotone Operator Theory in Hilbert Spaces}, CMS Books in Mathematics, Springer, New York (2011).

\bibitem{BelloCruzIusem} {\sc Bello Cruz, J. Y.; Iusem, A.N.}: \emph{Convergence of direct methods for paramonotone variational inequalities}, Comput. Optim. Appl. 46(2) (2010), pp. 247--263.

\bibitem{Bianchi} {\sc Bianchi, M., Hadjisavvas, N. and Schaible, S.}: \emph{On pseudomonotone maps $T$ for which $-T$ is also pseudomonotone}, J. Conv. Anal. 10 (2003), pp.  149--168.

\bibitem{BL06} {\sc Borwein, J. M. and Lewis, A. S.}: \emph {Convex Analysis and
Nonlinear Optimization: Theory and Examples} Springer Science and Business Media, New York (2006).

\bibitem{BC17}  {\sc Bo\c t, R. I. and Csetnek, E. R.}: \emph{Proximal-gradient algorithms for fractional programming}, Optimization 66(8) (2017), pp. 1383--1396.

%\bibitem{Cambini02}{\sc A. Cambini, J. P. Crouzeix, and L. Martein}: \emph{On the pseudoconvexity of a quadratic fractional function}, Optimization, 51(2002), pp.677--687.

%\bibitem{Cegielski12} {\sc A. Cegielski}: \emph{Iterative Methods for Fixed Point Problems in Hilbert Spaces}, Lecture Notes in Mathematics Vol. 2057, Springer Science and Business Media, Berlin (2012).

\bibitem{CengTeboulleYao} {\sc Ceng, L.C., Teboulle, M. and Yao, J.-C.}: \emph{Weak convergence of an iterative method for pseudomonotone variational inequalities and fixed-point problems}, J. Optim. Theory Appl. 146 (2010), pp. 19--31.

\bibitem{CensorGibali} {\sc Censor, Y., Gibali, A. and Reich, S.}: \emph{Extensions of Korpelevich's extragradient method for the variational inequality problem in Euclidean space}, Optimization 61(9) (2012), pp. 1119--1132.

\bibitem{Censor} {\sc Censor, Y., Gibali, A. and Reich, S.}:\emph{The subgradient extragradient method for solving variational inequalities in Hilbert space}, J. Optim. Theory Appl. 148 (2011), pp 318--335.

\bibitem{CottleFerland} {\sc Cottle, R.W. and Ferland, J.A.}: \emph{On pseudo-convex functions of nonnegative variables}, Math. Programm. 1 (1971), pp. 95--101.

\bibitem{CottleYao}  {\sc Cottle, R. W. and Yao, J. C.}: \emph{Pseudo-monotone complementarity problems in Hilbert space}, J. Optim. Theory Appl. 75 (1992), pp. 281--295.

\bibitem{DangLan}  {\sc Dang, C. D. and Lan, G.}: \emph{On the convergence properties of non-Euclidean extragradient methods for variational inequalities with
generalized monotone operators}, Comput. Optim. Appl. 60 (2015), pp. 277--310.

\bibitem{FacchineiPang03} {\sc Facchinei, F. and Pang, J.-S.}: \emph{Finite-Dimensional Variational Inequalities and Complementarity Problems}, Springer-Verlag, New York (2003).

\bibitem{HadjisavvasSchaibleWong} {\sc Hadjisavvas, N., Schaible, S. and Wong, N.-C.}: \emph{Pseudomonotone operators: a survey of the theory and its applications}, J. Optim. Theory Appl. 152 (2012), pp. 1--20.

\bibitem{HarkerPang90} {\sc Harker, P. T. and Pang, J.-S.}: \emph{A damped-Newton method for the linear complementarity problem}, in: E.L. Allower, K. Georg, Computational Solution of Nonlinear Systems of Equations, AMS Lectures on Applied Mathematics, Vol. 26 (1990), pp. 265--284.

\bibitem{KaramardianSchaible90} {\sc Karamardian, S. and Schaible, S.}: \emph{Seven kinds of monotone maps}, J. Optim. Theory Appl. 66 (1990) pp. 37-46.

\bibitem{Kim} {\sc Kim, D. S., Vuong, P. T. and Khanh, P. D.} :  \emph{Qualitative properties of strongly pseudomonotone variational inequalities},  Opt. Lett. 10 (2016), pp.1669--1679.

\bibitem{KinderlehrerStampacchia80} { \sc Kinderlehrer, D. and Stampacchia, G.}: \emph{An Introduction to Variational Inequalities and Their Applications}, Academic Press, New York (1980).

\bibitem{Kha} {\sc Khanh, P. D. and Vuong, P. T.}:  \emph{Modified projection method for strongly pseudomonotone variational inequalities},  J. Global Optim. 58 (2014),  pp.341--350. 

\bibitem{Korpelevich}
{\sc Korpelevich, G. M.}: \emph{The extragradient method for finding saddle points and other problems}, Ekonomika i Mat. Metody 12 (1976), pp.747--756.

\bibitem{Laszlo} {\sc L\'aszl\'o, S.C.}: \emph{Some existence results of solutions for general variational inequalities}, J. Optim. Theory Appl. 150 (2011), pp. 425--443.

\bibitem{Malitsky} {\sc Malitsky, Y.}:  \emph{Projected reflected gradient methods for monotone variational inequalities}, SIAM J. Optim. 25 (2015), pp. 502--520.

\bibitem{Malitsky2} {\sc Malitsky, Y.}:  \emph{Proximal extrapolated gradient methods for variational inequalities}, Optim. Meth. Softw. 33(1) (2018), pp. 140-164.

\bibitem{Malitsky3} {\sc Malitsky, Y.}:  \emph{Golden ratio algorithms for variational inequalities}, arXiv:1803.08832 (2018).

\bibitem{Mangasarian} {\sc Mangasarian O.L.}: \emph{Pseudo-convex functions}, SIAM J. Control Optim. 3 (1965), pp. 281-290.

\bibitem{MaugeriRaciti} {\sc Maugeri, A. and Raciti, F.}: \emph{On existence theorems for monotone and nonmonotone variational inequalities}, J. Conv. Analy. 16 (2009), pp. 899--911.

\bibitem{ShehuDongJiang} {\sc Shehu, Y., Dong, Q.-L. and Jiang, D.}: \emph{Single projection method for pseudo-monotone variational inequality in Hilbert spaces}, Optimization 68(1) (2019), pp.  385--409.

\bibitem{SolodovSvaiter99} {\sc Solodov, M. V. and Svaiter, B. F.}: \emph{A new projection method for variational inequality problems}, SIAM J. Control Optim. 37 (1999), pp. 765--776.

\bibitem{SolodovTseng96} {\sc Solodov, M. V. and Tseng, P.}: \emph{Modified projection-type methods for monotone variational inequalities},
SIAM J. Control Optim. 34 (1996), pp. 1814--1830. 

{\bibitem{ThongShehuIyiola} {\sc Thong, D.V., Shehu, Y. and Iyiola, O.S.}: \emph{Weak and strong convergence theorems for solving pseudo-monotone variational inequalities with non-Lipschitz mappings}, Numerical Algorithms, DOI: 10.1007/s11075-019-00780-0}

\bibitem{Tseng2000}
{\sc Tseng, P.}: \emph{A modified forward-backward splitting method for maximal monotone mappings}, SIAM J. Control Optim. 38 (2000), pp. 431--446.

\bibitem{Vuong} {\sc Vuong, P. T.}: \emph{On the weak convergence of the extragradient method for solving variational inequalities}, J. Optim. Theory Appl. 176(2) (2018), pp. 399--409. 

\bibitem{ZhuMarcotte96} {\sc Zhu, D. L. and Marcotte, P.}: \emph{Co-coercivity and its role in the convergence of iterative schemes for solving variational inequalities}, SIAM J. Control Optim. 6 (1996), pp. 714--726.



\end{thebibliography}
\end{document}